\documentclass[11pt]{amsart}
\usepackage{geometry}
\usepackage{amsmath,amssymb,amsthm,mathtools, xcolor}
\colorlet{RED}{red}
\usepackage{hyperref}
\usepackage{tikz}

\def\R{\mathbb{R}}

\def\Ss{\mathbb{S}}
\def\Co{\mathbb{C}}
\def\Hc{\mathcal{H}}
\def\Lc{\mathcal{L}}
\def\Gg{\mathfrak{G}}
\def\Z{\mathbb{Z}}

\numberwithin{equation}{section}

\newtheorem{theorem}{Theorem}[section]

\newtheorem{definition}[theorem]{Definition}
\newtheorem{remark}[theorem]{Remark}
\newtheorem{proposition}[theorem]{Proposition}
\newtheorem{lemma}[theorem]{Lemma}

\renewcommand{\d}{\,{\rm d}}

\newcommand{\al}{\alpha}
\newcommand{\la}{\lambda}

\usepackage{fullpage} 
\usepackage{setspace}
\usepackage{adjustbox}

\mathtoolsset{showonlyrefs}

\title{Gaussians Do Not Always Maximize Mixed-Norm \\ Strichartz Inequalities for the Schrödinger Equation}

\author{Felipe Gon\c{c}alves}
\address{
IMPA -- Instituto de Matem\'atica Pura e Aplicada, Estrada Dona Castorina 110, Rio de Janeiro, RJ 22460-320, Brazil.}
\email{goncalves@impa.br}

\author{Giuseppe Negro}
\author{Diogo Oliveira e Silva}
\address{
Center for Mathematical Analysis, Geometry and Dynamical Systems \&
Departamento de Matem\'atica\\
Instituto Superior T\'ecnico\\
Universidade de Lisboa\\
Av. Rovisco Pais\\
1049-001 Lisboa, Portugal.}
\email{giuseppe.negro@tecnico.ulisboa.pt}
\email{diogo.oliveira.e.silva@tecnico.ulisboa.pt}


\begin{document}

\begin{abstract}
We investigate the  maximization problem for the  family of mixed-norm Strichartz inequalities for the  Schrödinger equation,
$\|e^{-it\Delta/2}f\|_{L_t^qL_{\boldsymbol{x}}^r(\R^{1+d})}\le C_{q, r}\lVert f\rVert_{L^2(\mathbb R^d)}$, with
\[
\frac2q+\frac dr= \frac d2,
     \qquad
     q,r\geq 2,
     \qquad
   \text{and thus } \qquad r \leq \frac{2d}{d-2} \ \ (\text{if } d\geq 3).
 \]
We show that, in low dimensions $1\le d\le 5$, the thresholds  
\begin{equation}
\rho_1=10, \quad
\rho_2=6, \quad
\rho_3=4\sqrt{7}-6 \, \approx 4.583, 
\quad
\rho_4 = 2\sqrt{15}-4 \, \approx 3.746, 
\quad
\rho_5=\frac{10}{3}    \, \approx 3.333
\end{equation}
are such that 
gaussians are stable local  maximizers for $2<r<\rho_d$,  and fail to be local maximizers
for $\rho_d < r \leq \frac{2d}{d-2}$ (with the conventions there is no upper bound on $r$ when $d\in\{1,2\}$ and that $r=\infty$ is excluded when $d=2$).  In the cases $(q,r,d)\in\{(6,6,1),(8,4,1),(4,4,2)\}$, we establish global stability inequalities with effective stability constants. Both proofs hinge on spectral gaps which we compute exactly.
\end{abstract}

\maketitle

\section{Introduction}
The mixed-norm  Strichartz inequality for the free Schr\"odinger equation,
\begin{equation}\label{eq_freeSchr}
    iu_t=\frac12\Delta u, \quad u|_{t=0}=f,
\end{equation}
asserts the finiteness of the operator norm
\begin{equation}\label{eq_Strichartz}
{\bf S}_{q,r}:=   \sup_{0\neq f\in L^2(\R^d)} \frac{\|e^{-it\Delta/2}f\|_{L_t^qL_{\boldsymbol x}^r(\R^{1+d})}}{\|f\|_{L^2(\R^d)}} <\infty
\end{equation}
whenever the Lebesgue exponents $1\leq q,r\leq\infty$ and the spatial dimension $d\geq 1$ satisfy
\begin{equation}\label{eq_admissible}
    \frac2q+\frac dr=\frac d2,
    \qquad
     q,r\geq 2,
\end{equation}
with the usual exclusion of the endpoint $(q,r,d)=(2,\infty,2)$.
When  $q=r$, this was established in foundational work of Strichartz \cite{St77} drawing the connection to the restriction theorems of Stein \cite{St93} and Tomas \cite{To75}. The mixed-norm nonendpoint case, including $(q,r,d)=(4,\infty,1)$, was then addressed by Ginibre--Velo \cite{GV92} and Yajima \cite{Ya87}, while the endpoint $(q,r)=(2,\frac{2d}{d-2})$ was later established in dimensions $d\geq 3$ by Keel--Tao \cite{KT98}; see also \cite[\S 2.3, Theorem 2.3]{Tao06} for the precise formulation given here. The remaining endpoint $(q,r)=(\infty,2)$ follows directly from unitarity. Pairs $(q,r)$ for which \eqref{eq_admissible} holds are called {\it Schrödinger-admissible}.

We are interested in sharp and sharpened forms of the Strichartz inequality \eqref{eq_Strichartz}.
In view of Lieb's well-known principle that {\it gaussian kernels have only gaussian maximizers} \cite{Li90}, it is tempting to conjecture that gaussians maximize \eqref{eq_Strichartz} for all Schrödinger-admissible pairs. This conjecture was explicitly proposed in the pure norm case $q=r$ by Hundertmark--Zharnitsky  \cite[Conjecture 1.7]{HZ06} and in the mixed-norm case $q\neq r$ by Gonçalves  \cite[Conjecture 1]{Go19}.  Very recently, further numerical evidence for the general conjecture  was provided via a simple neural-network-based pipeline in \cite{FMV26}.

 In the present context, {\it gaussians} are the elements of the family
\begin{equation}\label{eq_gaussian_family_intro}
    \Gg
    :=
    \left\{
    c\,e^{-z|\boldsymbol x-\boldsymbol x_0|^2}e^{i\boldsymbol x\cdot \boldsymbol\xi_0}
    :
    c\in\Co\setminus\{0\},\;
    \Re z>0,\;
    \boldsymbol x_0,\boldsymbol\xi_0\in\R^d
    \right\}\subset L^2(\R^d).
\end{equation}
In Appendix~\ref{app_symmetry}, we explain  why  $\Gg$ is a manifold of real dimension $2d+4$  which coincides with the family generated from the standard gaussian in $\R^d$,
\begin{equation}\label{eq_stdgaussian}
g_0(\boldsymbol x):=e^{-|\boldsymbol x|^2/2},    
\end{equation}
by the symmetries of the  Schrödinger equation \eqref{eq_freeSchr}.
We  define the deficit functional $\delta_{q,r}$,
\begin{equation}\label{eq:two_homog_deficit}
    \delta_{q,r}[f]
    :=
    \frac{\|e^{-it\Delta/2}g_0\|_{L_t^qL_{\boldsymbol x}^r(\R^{1+d})}^2}
         {\|g_0\|_{L^2(\R^d)}^2}
    \|f\|_{L^2(\R^d)}^2
    -
    \|e^{-it\Delta/2}f\|_{L_t^qL_{\boldsymbol x}^r(\R^{1+d})}^2,
\end{equation}
where $g_0$ is given by \eqref{eq_stdgaussian}.
Observe that $\delta_{q,r}[g_0]=0$, and that $\delta_{q,r}[f]\geq 0$ for every $f\in L^2$ if and only if gaussians maximize \eqref{eq_Strichartz}. We further define the distance to the manifold $\Gg$ as follows:
\begin{equation}\label{eq_dist}
    \textup{dist}(f,\Gg):=\inf_{g\in\Gg}\|f-g\|_{L^2(\R^d)}.
\end{equation}

\begin{definition}
Let $(q,r)$ be a Schrödinger-admissible pair. 
Gaussians are said to be \emph{local maximizers} of
\eqref{eq_Strichartz} if there exists $\varepsilon>0$ such that
$\delta_{q,r}[f]\geq 0$
whenever \begin{equation}\label{eq_closetogaussians}
\textup{dist}(f,\Gg)<\varepsilon\|f\|_{L^2(\R^d)}.    
\end{equation} 
Furthermore, we say  that gaussians are \emph{stable local maximizers} of
\eqref{eq_Strichartz} if there exist $c,\varepsilon>0$ such that
$\delta_{q,r}[f] \geq  c\, \textup{dist}^2(f,\Gg)$
whenever \eqref{eq_closetogaussians} holds.
\end{definition}

\subsection{Main results}
The present paper disproves, in certain ranges of exponents, the conjecture that gaussians always maximize mixed-norm Strichartz inequalities for the Schrödinger equation. Our work also leaves open the possibility that gaussians are global extremizers in the remaining range.
In particular, our first main result identifies the precise  threshold at which gaussians cease to be stable local maximizers in low dimensions.
The precise statement hinges on the following thresholds: 
\begin{equation}
\rho_1=10, \quad
\rho_2=6, \quad
\rho_3=4\sqrt{7}-6, 
\quad
\rho_4 = 2\sqrt{15}-4, 
\quad
\rho_5=\frac{10}{3}.    
\end{equation}

\begin{theorem}\label{thm_main}
 Let $(q,r)$ be a Schrödinger-admissible pair. Then: 
\begin{itemize}
\item[(a)] If $1\leq d\leq 5$ and $2<r<\rho_d$, or if $d\geq 6$ and $2<r\leq\frac{2d}{d-2}$, then gaussians are stable local maximizers of \eqref{eq_Strichartz}.
\item[(b)] 
If $d=1$ and $\rho_1<r\leq \infty$, or if $d=2$ and $\rho_2<r<\infty$, or if $d\in\{3,4\}$ and $\rho_d<r\leq\frac{2d}{d-2}$, then gaussians are not local maximizers of \eqref{eq_Strichartz}.
\end{itemize}
\end{theorem}

\begin{remark}\label{rem_postmainthm}
    A few observations may help to further orient the reader.

\begin{itemize}
    \item[(a)]
If $(q,r)=(\infty,2)$, then every nonzero $L^2$ function is a maximizer; see \cite[Remark 1.9]{Sh09}.
\item[(b)] The upper bound $r \leq \frac{2d}{d-2}$ follows from \eqref{eq_admissible}  as long as $d\geq 3$. If $d\in\{1,2\}$, then this upper bound  boils down to  $r\leq\infty$ (with the exclusion of $r=\infty$ if $d=2$). 

\item[(c)]  In low dimensions $1\leq d\leq 5$, the threshold $\rho_d$ is the unique exponent not covered by Theorem \ref{thm_main}.  In order to establish whether gaussians are local maximizers for the threshold inequality, one would have to study higher-order Fréchet derivatives of $\delta_{q, \rho_d}$, which we do not pursue  here. If $d=5$, then the endpoint $r=\frac{2d}{d-2}$ coincides with the  threshold $\rho_5=\frac{10}{3}$. If $d\geq 6$, then gaussians are always stable local maximizers, and the picture is complete.
\item[(d)] The diagonal case $q=r$ corresponds to $r=2+\frac4d$. Theorem~\ref{thm_main} (a) thus recovers the previous positive results  from \cite[Theorem 1.2]{GN22}. The novelty here is that, away from the diagonal,  a genuinely new phenomenon occurs in dimensions $1\leq d\leq 4$, which in particular  disproves \cite[Conjecture 1]{Go19} in a range of exponents, but leaves open the possibility that the conjecture holds in the remaining range. For instance, when $d\geq 6$, gaussians are local maximizers for every nontrivial Schrödinger-admissible pair.

\end{itemize}
\end{remark}

Our second main result establishes effective global, or {\it constructive}, stability for all  Schrödinger-admissible pairs $(q,r)$ where $q,r$ are even integers and  $r$ divides $q$.\footnote{There exist two additional Schrödinger-admissible  {\it even} pairs  $(q,r)\in
\{(2,6),(2,4)\}$, which are endpoints in dimensions $d=3,4$, respectively. In view of Theorem \ref{thm_main} (b), gaussians fail to be  local maximizers in either case.} 

\begin{theorem}[Effective Global Stability]
\label{thm:effective_stability}
Let $f\in L^2(\mathbb R^d)$ with $f\geq 0$. If $d=1$, then
\begin{equation}\label{eq_stable66}
    \begin{split}
        \delta_{6,6}[f]
        &=
        \frac{1}{\sqrt[6]{3}}\lVert f\rVert_{L^2(\mathbb R)}^2
        -
        \lVert e^{-it\Delta/2}f\rVert_{L^6_tL^6_x(\mathbb R\times \mathbb R)}^2\\
        &\geq
        \frac{2}{8957979\sqrt[6]{3}}
        \operatorname{dist}^2(f,\Gg).
    \end{split}
\end{equation}
Moreover,
\begin{equation}\label{eq:Eight_Four_Stability}
    \begin{split}
        \delta_{8,4}[f]
        &=
        \frac{1}{\sqrt[4]{2}}\lVert f\rVert_{L^2(\mathbb R)}^2
        -
        \lVert e^{-it\Delta/2}f\rVert_{L^8_tL^4_x(\mathbb R\times \mathbb R)}^2\\
        &\geq
        \frac{9}{33554504\sqrt[4]{2}}
        \operatorname{dist}^2( f,\Gg).
    \end{split}
\end{equation}
If $d=2$, then 

\begin{equation}\label{eq_stable44}
    \begin{split}
        \delta_{4,4}[f]
        &=
        \frac{1}{\sqrt{2}}\lVert f\rVert_{L^2(\mathbb R^2)}^2
        -
        \lVert e^{-it\Delta/2}f\rVert_{L^4_tL^4_{\boldsymbol{x}}(\mathbb R\times\mathbb R^2)}^2\\
        &\geq
        \frac{1}{307212\sqrt{2}}
        \operatorname{dist}^2(f,\Gg).
    \end{split}
\end{equation}
For arbitrary complex-valued $f$, the same estimates hold with $\operatorname{dist}(f,\mathfrak G)$ replaced by $\operatorname{dist}(|f|,\mathfrak G)$.
\end{theorem}

Inequalities \eqref{eq_stable66} and \eqref{eq_stable44} were established in \cite[Theorem 1.2]{GN22} without any quantitative control on the stability constants. A similar situation underlies the sharpened Hausdorff--Young inequality of Christ \cite{Ch14}, for which constructive stability is still an open question.
For the Sobolev and log-Sobolev inequalities, constructive stability was recently established, with optimal dimensional dependence, by Dolbeault--Esteban--Figalli--Frank--Loss \cite{DEFFL25}.
The novelty of Theorem \ref{thm:effective_stability} lies in the off-diagonal case \eqref{eq:Eight_Four_Stability} together with the effective constants for all three inequalities.

\subsection{Proofs via spectral gap analysis}
The proofs of Theorems \ref{thm_main} and \ref{thm:effective_stability} rely on the exact computation of a spectral gap, which is interesting in its own right. We proceed to describe it in Theorem \ref{thm_gap} below.

Since $e^{-it\Delta/2}g_0$ never vanishes (see \eqref{eq_PropGauss}), the deficit functional
$\delta_{q,r}$ is twice real-differentiable at $g_0$ for
$q,r\in(2,\infty)$. For every $f\in L^2$, we denote the first two
Fr\'echet derivatives by
\[\delta_{q,r}'[g_0](f):=\left.\tfrac{\partial}{\partial\varepsilon}\right\vert_{\varepsilon=0}\delta_{q,r}[g_0+\varepsilon f],\,\text{ and }\,\delta_{q,r}''[g_0](f,f):=\left.\tfrac{\partial^2}{\partial\varepsilon^2}\right\vert_{\varepsilon=0}\delta_{q,r}[g_0+\varepsilon f].\]
 The local stability problem is reduced to the study of perturbations $g_0+h$ with
$h\perp T_{g_0}\Gg$; see Lemma~\ref{lem:third_estimate} below. Indeed, for $f$ sufficiently close to
$\Gg$, after applying a symmetry we may write
\[
    f=g_0+h,\qquad
    \|h\|_{L^2}=\operatorname{dist}(f,\Gg),
    \qquad h\perp T_{g_0}\Gg.
\]
Since the deficit is twice Fr\'echet differentiable at $g_0$
(for $q=2$, this follows directly from the calculations in
Section~2, where $\mathfrak d_{2,r}=\delta_{2,r}$), we have
\[
    \delta_{q,r}[g_0+h]
    =\frac12\delta_{q,r}''[g_0](h,h)
      +o(\|h\|_{L^2}^2).
\]
Thus gaussians are stable local maximizers provided the {\it spectral gap}
\begin{equation}\label{eq_SpecGap}
\Lambda_{q,r}(d):=\inf_{f\perp T_{g_0}\Gg}
\frac{\delta''_{q,r}[g_0](f,f)}
     {2\|f\|_{L^2(\mathbb R^d)}^2}
\end{equation}
is strictly positive.
Here, $T_{g_0} \Gg$ denotes the tangent space
to $\Gg$ at $g_0$, which is computed in Appendix~\ref{app_symmetry} and explicitly given by 
\[T_{g_0} \Gg:=\textup{span}_{\Co}\{e^{-|\boldsymbol{x}|^2/2},x_1e^{-|\boldsymbol{x}|^2/2},x_2e^{-|\boldsymbol{x}|^2/2},\ldots,x_de^{-|\boldsymbol{x}|^2/2},|\boldsymbol{x}|^2e^{-|\boldsymbol{x}|^2/2}\};\]
see also \cite[\S4]{GN22}.
Given $d\geq 2$, let $Y_m$ be a spherical harmonic of degree $m$ normalized by
\[
\int_{\Ss^{d-1}}|Y_m(\boldsymbol{\omega})|^2\d\sigma(\boldsymbol{\omega})=1,
\]
where $\d\sigma$ denotes the usual surface measure on the unit sphere $\mathbb S^{d-1}\subset\R^d$, satisfying $\sigma(\mathbb S^{d-1})=\frac{2\pi^{d/2}}{\Gamma(d/2)}$.
Further setting $\nu:=
\frac d2-1$, we define
\begin{equation}\label{eq_fkm_intro}
    f_{k,m}(\rho\boldsymbol{\omega})
    :=
    L_k^{\nu+m}(\rho^2)Y_m(\boldsymbol{\omega})\rho^m e^{-\rho^2/2},
    \qquad
    k,m\in \Z_{\geq 0},
\end{equation}
where $L_k^\alpha$ is the generalized Laguerre polynomial, normalized so that $L_k^\al(0)=\binom{k+\al}k$.
The changes needed to make this work in dimension $d=1$ are discussed in Remark \ref{rem_d1} below. For now, it suffices to mention that, when $d=1$, the function $f_{1,1}$ is proportional to $h_3$, the third Hermite function.
We further define the polynomials $p_{1,1}, p_{0,2}$ and $p_{2,0}$ via
$$p_{1,1}(\la) :=1- (1-\la)((\tfrac d2+2)\la-\tfrac d2),$$ $$p_{0,2}(\la):=\la,$$
\begin{align*}
p_{2,0}(\la) &:=1-(1-\la) \bigg( \left(\tfrac{3d}{4}
 + \tfrac{5}{2}\right) \la^2
 - \tfrac{d}{2}\la
 + \left(\tfrac{1}{2}-\tfrac{d}{4}
 \right) \bigg).
\end{align*}
These polynomials are part of a larger family $\{p_{k,m}\}$ which satisfies the following  property established in Proposition \ref{prop_modes} below:
\[
    \frac{\delta_{q,r}''[g_0](f_{k,m},f_{k,m})}{2\lVert f_{k, m}\rVert_{L^2(\R^d)}^2}\cong_{d,r} 1-p_{k,m}(\tfrac2r).
\]

A quantitative analysis of Theorem \ref{thm_main}  (a) leads to the following result.

\begin{theorem}[Spectral Gap]\label{thm_gap}
Let $f_{k,m}$ be as in \eqref{eq_fkm_intro}, and let $(q,r)$ be a Schrödinger-admissible pair.

\begin{itemize}
\item[(a)] If $d=1$ and $2<r<10$, then
\begin{align}\label{eq_gapd1}
\Lambda_{q,r}(1)
&=
\frac{\delta_{q,r}''[g_0]
(f_{1,1},f_{1,1})}
{2\|f_{1,1}\|_{L^2(\R)}^2} \\
&=
\pi^{\frac{2}{q}-\frac{1}{2}}
\left(\frac{2\pi}{r}\right)^{\frac1r}
\left(1-p_{1,1}\left(\frac{2}{r}\right)\right)>0.
\end{align}

\item[(b)] If $d=2$ and $2<r<6$, then
\begin{align}\label{eq_gapd2}
\Lambda_{q,r}(2)
&=
\min\left\{
\frac{\delta_{q,r}''[g_0]
(f_{1,1},f_{1,1})}
{2\|f_{1,1}\|_{L^2(\R^2)}^2},
\frac{\delta_{q,r}''[g_0]
(f_{0,2},f_{0,2})}
{2\|f_{0,2}\|_{L^2(\R^2)}^2}
\right\} \\
&=
\pi^{\frac{2}{q}-1}
\left(\frac{2\pi}{r}\right)^{\frac2r}
\min\left\{
1-p_{1,1}\left(\frac{2}{r}\right),
1-p_{0,2}\left(\frac{2}{r}\right)
\right\}>0.
\end{align}

\item[(c)] If $d\in\{3,4,5\}$ and $2<r<\rho_d$, or if
$d\geq6$ and $2<r\leq\frac{2d}{d-2}$, then
\begin{align}\label{eq_gapd3}
\Lambda_{q,r}(d)
&=
\min\left\{
\frac{\delta_{q,r}''[g_0]
(f_{1,1},f_{1,1})}
{2\|f_{1,1}\|_{L^2(\R^d)}^2},
\frac{\delta_{q,r}''[g_0]
(f_{0,2},f_{0,2})}
{2\|f_{0,2}\|_{L^2(\R^d)}^2},
\frac{\delta_{q,r}''[g_0]
(f_{2,0},f_{2,0})}
{2\|f_{2,0}\|_{L^2(\R^d)}^2}
\right\} \\
&=
\pi^{\frac{2}{q}-\frac{d}{2}}
\left(\frac{2\pi}{r}\right)^{\frac dr}
\min\left\{
1-p_{1,1}\left(\frac{2}{r}\right),
1-p_{0,2}\left(\frac{2}{r}\right),
1-p_{2,0}\left(\frac{2}{r}\right)
\right\}>0.
\end{align}
\end{itemize}
\end{theorem}

The minima in the  spectral gaps computed in Theorem \ref{thm_gap} cannot be removed. In fact, numerical evidence indicates that, in every dimension $d\geq 2$, the minimizing polynomial changes as $r$ varies in the admissible range; see Figure~\ref{fig_alt}, and \cite{NOS23} for different but related switching phenomena.

\begin{figure}[h]
    \centering
    \includegraphics[width=0.49\linewidth]{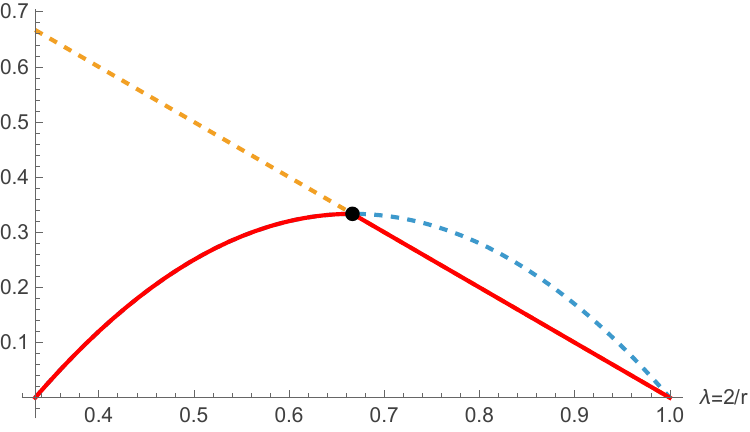}
    \includegraphics[width=0.5\linewidth]{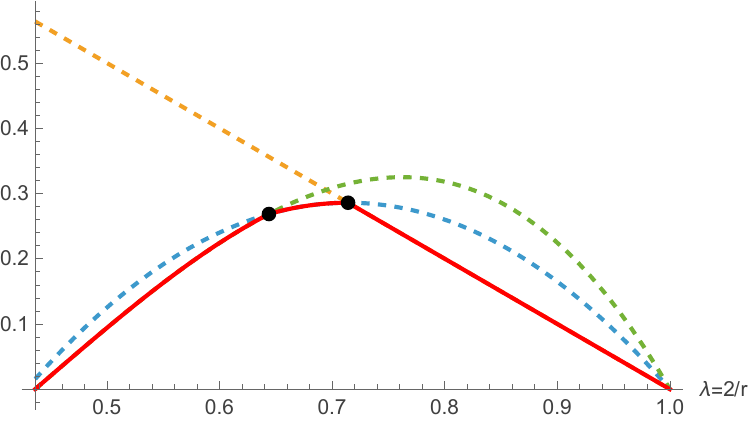}
    \caption{The cases $d\in\{2,3\}$ of Theorem \ref{thm_gap}. 
    Left: The case $d=2$ for $\frac13<\lambda=\frac2r<1$, with $1-p_{1,1}$ in dashed  blue, $1-p_{0,2}$ in dashed orange, and min$\{1-p_{1,1},1-p_{0,2}\}$ in solid red.
    Right: The case $d=3$ for $\frac1{2\sqrt7-3}<\lambda=\frac2r<1$, with $1-p_{1,1}$ in dashed blue, $1-p_{0,2}$ in dashed orange, $1-p_{2,0}$ in dashed green, and min$\{1-p_{1,1},1-p_{0,2},1-p_{2,0}\}$ in solid red. 
Note that the two panels display different $\lambda$-ranges. When $d=3$, we already observe the two switches which appear to be present in all  $d\geq 3$.}
    \label{fig_alt}
\end{figure}

We record  the spectral gaps of the cases corresponding to Theorem~\ref{thm:effective_stability}:
\begin{equation}\label{eq:spectral_gaps_concrete}
    \Lambda_{6,6}(1)=\frac{2}{9\sqrt[6]{3}},
    \qquad
    \Lambda_{8,4}(1)=\frac{3}{8\sqrt[4]{2}},
    \qquad
    \Lambda_{4,4}(2)=\frac{1}{4\sqrt{2}}.
\end{equation}
Here we used for $d=1$
\[
1-p_{1,1}\left(\frac13\right)=\frac29,
\qquad
1-p_{1,1}\left(\frac12\right)=\frac38,
\]
and, for $d=2$,
\[
1-p_{1,1}\left(\frac12\right)=\frac14,
\qquad
1-p_{0,2}\left(\frac12\right)=\frac12.
\]

We finish this section by briefly commenting on the proofs of Theorems \ref{thm_main} and \ref{thm:effective_stability}.

The idea for  Theorem \ref{thm_main} is that coercivity of $\delta_{q,r}''$ implies stability of $\delta_{q,r}$.
The proof proceeds through a spectral decomposition of the Hessian at
$g_0$, which we accomplish via the lens transform.
As part of our proof we will show that, when $r=\rho_d$ and $1\le d\le5$, the second variation of the deficit is positive semidefinite on $(T_{g_0}\Gg)^\perp$. On this space, its kernel has complex dimension $1$, except in dimension $d=2$, where it has complex dimension $2$; see Remark \ref{rem_dimKer} below.

 The method for Theorem~\ref{thm:effective_stability} uses a local-to-global argument due to Christ~\cite{Ch17}, as adapted in \cite[p.~10]{DEFFL25}, combining our spectral gap analysis from Theorem~\ref{thm_gap} with the heat-flow monotonicity of Strichartz norms identified in \cite{BBCH09}. The second inequality~\eqref{eq:Eight_Four_Stability} is especially interesting since it is mixed-norm, and treating the mixed-norm case thoroughly is exactly the main purpose of the present paper.

\subsection{Contextual remarks}\label{sec_context}
The quest for sharp Strichartz inequalities for the Schrödinger equation has a long and rich history,  intertwined with that of sharp restriction theory \cite{FOS17, NOST23} in a beautifully intricate way.
One can broadly distinguish four main types of contributions:
\begin{itemize}
    \item {\bf Existence of maximizers and precompactness of maximizing sequences modulo symmetries.} Kunze \cite{Ku03} established this when $(q,r,d)=(6,6,1)$, using a refinement of the concentration--compactness principle of Lions \cite{Li85}. In arbitrary dimensions $d\geq 1$, Shao \cite[Cor.\@ 1.8]{Sh09} established it for nonendpoint Schrödinger-admissible  pairs (i.e., $q\neq 2$), using a profile decomposition and refined Strichartz estimates due to Merle--Vega \cite{MV98}, Bégout--Vargas \cite{BV07}, and Carles--Keraani \cite{CK07}. Stovall \cite{St20} later treated $L^p\to L^q$ Fourier extension inequalities on the paraboloid when $p\neq 2$.
    \item {\bf Best constants and maximizers.} Ozawa--Tsutsumi \cite{OT98} computed the best constant \eqref{eq_Strichartz} when $(q,r,d)=(4,4,2)$, showing that ${\bf S}_{4,4}=2^{-1/4}$.  Hundertmark--Zharnitsky \cite{HZ06} and Foschi  \cite{Fo07} further computed ${\bf S}_{6,6}=3^{-1/12}$ and characterized the maximizers for $(q,r,d)\in\{(6,6,1),(4,4,2)\}$. The case $(q,r,d)=(8,4,1)$ was subsequently addressed by Carneiro \cite{Ca09}, who computed ${\bf S}_{8,4}=2^{-1/8}$. Alternative proofs followed, first by Bennett-Bez--Carbery--Hundertmark \cite{BBCH09} using quadratic heat flow, and then by Gonçalves \cite{Go19} using orthogonal polynomials and generating functions (he also formulated  \cite[Conjecture 1]{Go19} which we disprove here). A surprising detour was taken by Gonçalves--Zagier \cite{GZ22}, who found that {\it derivatives} of gaussians are the maximizers for \eqref{eq_Strichartz}  restricted to {\it odd} functions when $(q,r,d)=(6,6,1)$.
    \item {\bf Sharpened inequalities and stability.} Gonçalves \cite{Go19b} established a sharpened Strichartz inequality for radial functions when $(q,r,d)=(4,4,2)$. The work of Gonçalves--Negro \cite{GN22}, a natural precursor of the present work,  used the lens transform to prove that gaussians are local maximizers in the Strichartz--Stein--Tomas setting $q=r=2+\frac4d$.
    They further showed the existence of $C,\varepsilon$, depending only on $d$ but non-explicit, for which
    \begin{equation}\label{eq_stablest} \delta_{2+\frac4d,2+\frac4d}[f]\geq C\textup{dist}^2(f,\Gg) \text{ if }\textup{dist}(f,\Gg)<\varepsilon \|f\|_{L^2(\R^d)},
    \end{equation}
     noting that the latter condition can be removed in the lowest dimensional cases $d\in\{1,2\}$ (i.e., when the corresponding maximizers have been characterized). Recently, Di--Yan \cite{DY25} showed the existence of minimizers for the stability inequality~\eqref{eq_stablest}, under the assumption that gaussians maximize the corresponding sharp inequality. They also obtained a formula for the spectral gap  in the diagonal paraboloid setting; see \cite[Proposition~1.6]{DY25}.
    \item {\bf Applications.} 
    We mention three distinct examples, in chronological order, which emerged from the study of sharp Strichartz inequalities for the Schrödinger equation: Duyckaerts--Merle--Roudenko \cite{DMR11} proved some sharp results for the $L^2$-critical nonlinear Schrödinger equation, Frank--Lieb--Sabin \cite{FLS16}  established the conditional existence of endpoint Stein--Tomas maximizers on the sphere, and Oliveira e Silva--Quilodrán \cite{OSQ18} together with Brocchi--Oliveira e Silva--Quilodrán \cite{BOSQ20} resolved certain dichotomies from the PDE literature concerning  fractional and higher-order
Schrödinger equations.
\end{itemize}

Two further types of isolated contributions deserve to be mentioned. Wayne--Zharnitsky \cite{WZ21} and, more recently, Freire--Muñoz--Valenzuela \cite{FMV26} provided systematic numerical investigations of Strichartz maximizers. On the other hand, Christ--Quilodrán \cite{CQ14} showed that gaussians never extremize adjoint restriction inequalities for paraboloids outside the $L^2$ setting, providing an important counterpoint to the heuristic suggested by Lieb's principle \cite{Li90}. The present work grew out of a closer look at this tension.

\subsection{Outline}
In Section~\ref{sec_prelim}, we introduce the lens transform, compute the second variation, and diagonalize it in Laguerre/spherical-harmonic modes. In Section~\ref{sec_proofs}, we analyze the resulting polynomial eigenvalues and prove Theorems~\ref{thm_main} and~\ref{thm_gap}, treating the case $r=\infty$ when $d=1$ separately. In Section~\ref{sec:stability}, we combine the exact spectral gaps with heat-flow monotonicity to prove the effective global stability of Theorem \ref{thm:effective_stability}. The appendices contain some symmetry considerations and the relevant PARI/GP code.

\subsection{Notation}
 We use  $X\cong Y$ to denote the identity $X=CY$, for some $C>0$.

\section{Preliminaries}
\label{sec_prelim}
It is computationally simpler to consider the $q$-homogeneous version of the deficit functional:
\begin{equation}\label{eq:q_homog_deficit}
    \mathfrak{d}_{q,r}[f]
    :=
    \frac{\|e^{-it\Delta/2}g_0\|_{L_t^qL_{\boldsymbol{x}}^r(\R^{1+d})}^q}
         {\|g_0\|_{L^2(\R^d)}^q}
    \|f\|_{L^2(\R^d)}^q
    -
    \|e^{-it\Delta/2}f\|_{L_t^qL_{\boldsymbol{x}}^r(\R^{1+d})}^q.
\end{equation}
There is no conceptual difference between the local analysis of $\delta_{q,r}$ and of $\mathfrak{d}_{q, r}$;  Lemma~\ref{lem:deficit_dictionary} below elucidates  the precise relationship between the two. The first step in the proof of all results is  to apply a convenient transform to compactify the time integrals in $\mathfrak{d}_{q, r}$. 
\subsection{Lens transform and the second variation}
The lens transform is a map that is applied to the solutions of \eqref{eq_freeSchr}, turning the Laplace operator into the harmonic oscillator $\Delta-|\cdot|^2$. In order to apply it, 
we write the free Schr\"odinger equation as
\begin{equation}\label{eq_Schr}
    {\left(i\partial_t-\frac12\Delta_{\boldsymbol{x}}\right)u=0.}
\end{equation}

Define the lens map
\[
\Lc:\R^{1+d}\to\left(-\frac\pi2,\frac\pi2\right)\times\R^d,
\qquad
(t,\boldsymbol{x})\mapsto\left(\arctan t,\frac{\boldsymbol{x}}{\sqrt{1+t^2}}\right),
\]
with inverse $(t,\boldsymbol{x})=\left(\tan T,\frac{\boldsymbol{X}}{\cos T}\right)$. If $u=u(t,\boldsymbol{x})$ and
$U=U(T,\boldsymbol{X})$ are related by
\begin{equation}\label{eq_Uu}
    U(T,\boldsymbol{X})
    =
    \frac{1}{\cos^{d/2}T}
    u\!\left(\tan T,\frac{\boldsymbol{X}}{\cos T}\right)
    {e^{\frac{i}{2}|\boldsymbol{X}|^2\tan T}},
\end{equation}
then $U$ solves
\begin{equation}\label{eq:harmonic_oscillator_evolution}
{\left(i\partial_T-\frac12\Hc\right)U=0},
\qquad
\Hc:=\Delta_{\boldsymbol{X}}-|\boldsymbol{X}|^2,
\end{equation}
whenever $u$ solves \eqref{eq_Schr}; see \cite{Tao09}. In our notation, this reads as follows.
\begin{lemma}[Lens Transform]\label{lem_lens}
If $(t,\boldsymbol{x})=\left(\tan T,\frac{\boldsymbol{X}}{\cos T}\right)$, then
\[
(e^{-it\Delta/2}f)(\boldsymbol{x})
=
\cos^{d/2}T\,
(e^{-iT\Hc/2}f)(\boldsymbol{X})\,
{e^{-\frac{i}{2}|\boldsymbol{X}|^2\tan T}}.
\]
\end{lemma}

\begin{proposition}\label{prop_lensnorms}
Let $(q,r)$ be a Schrödinger-admissible pair. Then
\begin{equation}\label{eq_lensnorms}
    \|e^{-it\Delta/2}f\|_{L_t^qL_{\boldsymbol{x}}^r(\R^{1+d})}^q
    =
    \int_{-\pi/2}^{\pi/2}
    \left(
    \int_{\R^d}|(e^{-iT\Hc/2}f)(\boldsymbol{X})|^r\d \boldsymbol{X}
    \right)^{q/r}\d T,
\end{equation}
with the usual interpretation when $r=\infty$.
\end{proposition}

\begin{proof}
Lemma~\ref{lem_lens} yields
\[
\|e^{-it\Delta/2}f\|_{L_t^qL_{\boldsymbol{x}}^r(\R^{1+d})}^q
=
\int_\R
\left(
\int_{\R^d}
\cos^{dr/2}T\,
|(e^{-iT\Hc/2}f)(\boldsymbol{X})|^r
\d \boldsymbol{x}
\right)^{q/r}\d t,
\]
where $T=\arctan t$ and $\boldsymbol{X}=\boldsymbol{x}\cos T$. The change of variables
$\boldsymbol{x}=\boldsymbol{X}/\cos T$ produces the Jacobian $\d \boldsymbol{x}=\cos^{-d}T\d \boldsymbol{X}$, so the inner integral
becomes
\[
\cos^{d(r/2-1)}T\int_{\R^d}|(e^{-iT\Hc/2}f)(\boldsymbol{X})|^r\d \boldsymbol{X}.
\]
Raising to the $q/r$ power and using the Schrödinger admissibility~\eqref{eq_admissible} yields
\[
\|e^{-it\Delta/2}f\|_{L_t^qL_{\boldsymbol{x}}^r(\R^{1+d})}^q
=
\int_\R
\left(
\int_{\R^d}|(e^{-iT\Hc/2}f)(\boldsymbol{X})|^r\d \boldsymbol{X}
\right)^{q/r}
\cos^2 T\,\d t.
\]
Finally, the change of variables $t=\tan T$ gives
$\d t=\cos^{-2}T\d T$, and \eqref{eq_lensnorms} follows.
\end{proof}

Note that this time compactification  reduces the  time interval $\mathbb R$ to $(-\pi/2, \pi/2)$. This is still not ideal, as we would like to work on $(-\pi, \pi)$ in order to fully exploit orthogonality of the characters. Moreover, it will be convenient to shift the harmonic oscillator $\mathcal{H}/2$ by $d/2$, coming from the eigenvalue relation $\Hc g_0=-dg_0$ (recall that $g_0=e^{-|\cdot|^2/2}$ denotes  the standard gaussian in $\R^d$). This results in $-(\Hc+d)/2$, which is known as the \emph{number operator} since its eigenvalues coincide with the nonnegative integers. The following proposition implements both improvements by exploiting a hidden symmetry of the harmonic oscillator evolution~\eqref{eq:harmonic_oscillator_evolution}.\footnote{A similar idea works for the wave equation; see~\cite[Lemma~A.5]{NOSST24}.}
\begin{proposition}\label{prop:symmetry_trick}
For every $f\in L^2(\R^d)$,
\begin{equation}\label{eq_symmetry_trick}
    \|e^{-iT\Hc/2}f\|_{L_T^qL_{\boldsymbol{X}}^r((-\pi/2,\pi/2)\times\R^d)}^q=\frac12\|e^{-iT(\Hc+d)/2}f\|_{L_T^qL_{\boldsymbol{X}}^r((-\pi,\pi)\times\R^d)}^q.
\end{equation}
\end{proposition}
 The proof hinges on classical properties of the number operator; see e.g.~\cite[\S18.39(i)]{NIST:DLMF}.
\begin{lemma}\label{lem:number_operator_props}
    The operator $-(\Hc+d)/2$ has pure point spectrum $\mathbb Z_{\ge 0}$. Eigenfunctions are even or odd according to the parity of the corresponding eigenvalue. 
\end{lemma}

\begin{proof}[Proof of Proposition~\ref{prop:symmetry_trick}]
First observe that 
\begin{equation}\label{eq:phase_translation_harmless}
    \lvert e^{-iT(\Hc + d)/2}f(\boldsymbol{X})\rvert=\lvert e^{-iT d/2}e^{-iT\Hc/2}f(\boldsymbol{X})\rvert= \lvert e^{-iT\Hc/2}f(\boldsymbol{X})\rvert, 
\end{equation}
so the phase shift does not influence any of the norms in~\eqref{eq_symmetry_trick}. We claim that 
\[
U(T,\boldsymbol{X}):=e^{-iT(\Hc+d)/2}f(\boldsymbol{X})
\]
satisfies the symmetry 
\begin{equation}\label{eq:hidden_symmetry_U}
    U(T+\pi,-\boldsymbol{X})=U(T, \boldsymbol{X}).
\end{equation}
 It will suffice to prove~\eqref{eq:hidden_symmetry_U} when $f=f_n$ is an eigenfunction of the number operator $-(\Hc+d)/2$. Assuming $-\frac{\Hc + d}{2} f_n=n f_n$ for some $n\in\mathbb Z_{\ge 0}$,
\[
U(T,\boldsymbol{X})=e^{in T}f_n(\boldsymbol{X}),
\]
and since $f_n(-\boldsymbol{X})=(-1)^n f_n(\boldsymbol{X})$ by Lemma~\ref{lem:number_operator_props}, we immediately obtain 
\begin{equation*}
    U(T+\pi,-\boldsymbol{X})=e^{in(T+\pi)}f_{n}(-\boldsymbol{X})=e^{in T}f_{n}(\boldsymbol{X})=U(T, \boldsymbol{X})
    ,
\end{equation*} 
as claimed.
We can now conclude as follows:
\[
\begin{aligned}
\|U\|_{L_T^qL_{\boldsymbol{X}}^r((-\pi,\pi)\times\R^d)}^q
&=
\left(
\int_{-\pi}^{-\pi/2}
+
\int_{-\pi/2}^{0}
+
\int_0^{\pi/2}
+
\int_{\pi/2}^{\pi}
\right)
\left(
\int_{\R^d}|U(T,\boldsymbol{X})|^r\d \boldsymbol{X}
\right)^{q/r}\d T \\
&=
2
\left(
\int_{-\pi/2}^{0}
+
\int_0^{\pi/2}
\right)
\left(
\int_{\R^d}|U(T,\boldsymbol{X})|^r\d \boldsymbol{X}
\right)^{q/r}\d T,
\end{aligned}
\]
which is \eqref{eq_symmetry_trick}.
\end{proof}
We  now return  to the $q$-homogeneous deficit functional defined in~\eqref{eq:q_homog_deficit}.
Since $r>2$, the derivatives at $g_0$ exist:
\begin{equation}\label{eq:frak_derivatives}
    \begin{array}{cc}
        \mathfrak{d}_{q,r}'[g_0](f):=\left.\frac{\partial}{\partial \varepsilon}\right|_{\varepsilon=0} \mathfrak{d}_{q,r}[g_0+\varepsilon f], 
        & \mathfrak{d}_{q,r}''[g_0](f, f):=\left.\frac{\partial^2}{\partial \varepsilon^2}\right|_{\varepsilon=0} \mathfrak{d}_{q,r}[g_0+\varepsilon f].
    \end{array}
\end{equation}
The relationship between these derivatives and those of the 2-homogeneous deficit $\delta_{q,r}$ is as follows.
\begin{lemma}\label{lem:deficit_dictionary}
    It holds that
    \begin{align}
        \frac{\mathfrak{d}_{q,r}'[g_0]}{q\lVert e^{-it\Delta/2}g_0\rVert_{L_t^qL_{\boldsymbol{x}}^r(\R^{1+d})}^{q-1}} &= \frac{\delta_{q,r}'[g_0]}{2\lVert e^{-it\Delta/2}g_0\rVert_{L_t^qL_{\boldsymbol{x}}^r(\R^{1+d})}}=0\,; \label{eq:first_der_dict} \\ 
        \frac{\mathfrak{d}_{q,r}''[g_0]}{q\lVert e^{-it\Delta/2}g_0\rVert_{L_t^qL_{\boldsymbol{x}}^r(\R^{1+d})}^{q-1}}&= \frac{\delta_{q,r}''[g_0]}{2\lVert e^{-it\Delta/2}g_0\rVert_{L_t^qL_{\boldsymbol{x}}^r(\R^{1+d})}}. \label{eq:second_der_dict} 
    \end{align}
    Explicitly, \begin{equation}\label{eq:second_der_dict_explicit}
    \frac{\delta_{q,r}''[g_0]}{2}
    =
    \pi^{\frac{2-q}{q}}
    \left(\frac{2\pi}{r}\right)^{\frac{d(2-q)}{2r}}
    \frac{\mathfrak{d}_{q,r}''[g_0]}{q}.
\end{equation}
\end{lemma}
\begin{proof}
    Letting $F(\varepsilon):=\frac{\|e^{-it\Delta/2}g_0\|_{L_t^qL_{\boldsymbol{x}}^r}}
         {\|g_0\|_{L^2}}
    \|g_0+\varepsilon f\|_{L^2}$ and $G(\varepsilon):=\|e^{-it\Delta/2}(g_0+\varepsilon f)\|_{L_t^qL_{\boldsymbol{x}}^r}$, we have that $\mathfrak{d}_{q,r}[g_0+\varepsilon{f}]=F^q(\varepsilon)-G^q(\varepsilon)$ and $\delta_{q,r}[g_0+\varepsilon{f}]=F^2(\varepsilon)-G^2(\varepsilon)$. The first identity in~\eqref{eq:first_der_dict}  follows by differentiating in $\varepsilon$ and using $F(0)=G(0)$. The fact that such a derivative necessarily vanishes is contained in Christ--Quilodrán~\cite{CQ14}, but we will give a direct proof using the lens transform  in Lemma \ref{lem_second_variation} below. In particular, $F'(0)=G'(0)$, which then  implies~\eqref{eq:second_der_dict}. Expression~\eqref{eq:second_der_dict_explicit} is a consequence of $\left\|e^{-it\Delta/2}g_0\right\|_{L_t^qL_{\boldsymbol{x}}^r(\mathbb{R}^{1+d})}
    =
    \pi^{\frac1q}
    \left(\frac{2\pi}{r}\right)^{\frac{d}{2r}}$, which in turn follows from \eqref{eq_PropGauss}.
\end{proof}

\begin{lemma}\label{lem_second_variation}
Let $(q,r)$  be a Schrödinger-admissible pair with $2<r<\infty$. For every $f\in L^2(\mathbb R^d)$, 
\begin{equation}\label{eq:first_variation_vanishes}
    \mathfrak{d}_{q, r}'[g_0](f)=0.
\end{equation}
Moreover, if
\begin{equation}\label{eq_ortho}
    \int_{\R^d}f(\boldsymbol{X})g_0(\boldsymbol{X})\d \boldsymbol{X}=0,
\end{equation} then the following identity holds:
\begin{equation}\label{eq_second_variation}
    \begin{aligned}
    \mathfrak{d}_{q,r}''[g_0](f,f)
    =&\;
    \pi q\,
    \frac{\|g_0\|_{L^r(\R^d)}^q}{\|g_0\|_{L^2(\R^d)}^2}
    \|f\|_{L^2(\R^d)}^2 -
    \frac{qr}{4}\,
    \|g_0\|_{L^r(\R^d)}^{q-r}
    \int_{-\pi}^{\pi}\int_{\R^d}
    g_0^{r-2}(\boldsymbol{X})|(e^{-iT(\Hc+d)/2}f)(\boldsymbol{X})|^2
    \d \boldsymbol{X}\d T \\
    &-
    \frac{q(q-r)}{2}\,
    \|g_0\|_{L^r(\R^d)}^{q-2r}
    \int_{-\pi}^{\pi}
    \left(
    \Re\int_{\R^d}g_0^{r-1}(\boldsymbol{X})(e^{-iT(\Hc+d)/2}f)(\boldsymbol{X})\d \boldsymbol{X}
    \right)^2
    \d T.
    \end{aligned}
\end{equation}
\end{lemma}

\begin{proof}
Using Propositions~\ref{prop_lensnorms} and \ref{prop:symmetry_trick}, we obtain
\begin{equation}\label{eq:deficit_post_lens}
\mathfrak{d}_{q,r}[f]
=
\frac12
\frac{\|e^{-iT(\Hc+d)/2}g_0\|_{L_T^qL_{\boldsymbol{X}}^r((-\pi,\pi)\times\R^d)}^q}
     {\|g_0\|_{L^2(\R^d)}^q}
\|f\|_{L^2(\R^d)}^q
-
\frac12
\|e^{-iT(\Hc+d)/2}f\|_{L_T^qL_{\boldsymbol{X}}^r((-\pi,\pi)\times\R^d)}^q.
\end{equation}
Since  $\mathfrak{d}_{q,r}'[g_0](g_0)=0$, it will suffice to prove~\eqref{eq:first_variation_vanishes} for an eigenfunction $f=f_n$ satisfying $-\frac{\mathcal H+d}{2}f_n=n f_n$, for some $n\in\mathbb Z_{>0}$. In this case,~\eqref{eq:first_variation_vanishes} is a consequence of
\begin{equation}\label{eq:first_var_vanishes_proof}
    \begin{aligned}
    &\left.\frac{\partial}{\partial \varepsilon}\right|_{\varepsilon=0}
    \int_{-\pi}^{\pi}
    \left(\int_{\mathbb R^d}
    \left|g_0+\varepsilon e^{inT}f_n\right|^r\,\d \boldsymbol{X}\right)^{q/r}\,\d T \\
    &\qquad =q\|g_0\|_{L^r}^{q-r}
    \Re\int_{-\pi}^{\pi}\int_{\mathbb R^d}
    g_0^{r-1}(\boldsymbol{X})e^{inT}f_n(\boldsymbol{X})\,\d \boldsymbol{X}\,\d T=0,
    \end{aligned}
\end{equation}
where the last identity follows from $\int_{-\pi}^{\pi} e^{inT}\, \d T=0$. 

We now turn to the proof of~\eqref{eq_second_variation}. Here, we consider a general $f$ satisfying~\eqref{eq_ortho}, not necessarily a single eigenfunction. The second derivative of the first summand of~\eqref{eq:deficit_post_lens} is simply
\[
\pi q\,
\frac{\|g_0\|_{L^r(\R^d)}^q}{\|g_0\|_{L^2(\R^d)}^2}
\|f\|_{L^2(\R^d)}^2.
\]
For the second summand, let $U(T,\boldsymbol{X})=(e^{-iT(\Hc+d)/2}f)(\boldsymbol{X})$, and expand
\[
|g_0+\varepsilon U|^r
=
g_0^r
+ r\varepsilon g_0^{r-1}\Re U
+ \frac{\varepsilon^2}{2}g_0^{r-2}
\left(
\frac{r^2}{2}|U|^2+\frac{r(r-2)}{2}\Re(U^2)
\right)
+ o(\varepsilon^2).
\]
Inserting this into the mixed norm and differentiating twice yields three terms.
The one containing $\Re(U^2)$ vanishes after integration in $T$, because
\eqref{eq_ortho} removes the ground state from the expansion of $f$ into eigenfunctions of the number operator $-(\Hc +d)/2$. Therefore all summands contain exponentials $e^{i(n_1+n_2)T}$ with $n_1, n_2\in\mathbb Z_{>0}$, which integrate to zero on $(-\pi,\pi)$. The remaining two terms coincide with the last two summands
in \eqref{eq_second_variation}.
\end{proof}

\subsection{Laguerre and spherical-harmonic modes}\label{subsection:laguerre}

Recall the notation
$\nu:=\frac d2-1$.
For $d\geq 2$, let $\{Y_{m, j}\}$ denote a complete orthogonal system of spherical harmonics of degree $m$ normalized via
\[
\int_{\Ss^{d-1}}|Y_{m, j}(\boldsymbol{\omega})|^2\d\sigma(\boldsymbol{\omega})=1.
\]
We then define
\begin{equation}\label{eq_fkm}
    f_{k,m,j}(\rho\boldsymbol{\omega})
    :=
    L_k^{\nu+m}(\rho^2)Y_{m,j}(\boldsymbol{\omega})\rho^m e^{-\rho^2/2},
    \qquad
    k,m\in \Z_{\geq 0}, \, j\in\{1,\ldots,N(d,m)\}.
\end{equation}
Here, $N(d, m):=\binom{m+d-1}{m}-\binom{m+d-3}{m-2}$ denotes the degeneracy of the spherical harmonics of degree $m$, where the second binomial coefficient is understood to vanish for $m<2$. Moreover, $L_k^\alpha$ denotes the generalized Laguerre polynomial, normalized so that $L_k^\al(0)=\binom{k+\al}k$, or equivalently,
$$
\int_0^\infty L_k^\al(\rho)^2 \rho^\al e^{-\rho} \d \rho = \al!L_k^\al(0) = \al!\binom{k+\al}k,
$$
where $\al!:=\Gamma(\al+1)$;
see \cite{AIK}. In particular, we obtain for $\alpha:=\nu+m$ $$\|f_{k,m,j}\|_{L^2(\R^d)}^2= \frac{1}{2}\int_0^\infty L_k^{\al}(s)^2 s^{\al} e^{-s}\, \d s = \frac{\al!}2 L_{k}^\al(0) .$$
These functions form a complete system of
eigenfunctions for the number operator $-(\Hc+d)/2$, satisfying
\begin{equation}\label{eq_fkm_eigen}
    e^{-iT(\Hc+d)/2}f_{k,m,j}=e^{i(2k+m)T}f_{k,m,j}.
\end{equation}

\begin{remark}\label{rem_d1}
The case $d=1$ can be described via the same notation if one endows
$\Ss^0=\{-1,+1\}$ with the counting measure assigning mass $1$ to each
point. Then the only spherical harmonics are
\[
Y_0\equiv 1,
\qquad
Y_1(\omega)=\omega,
\]
corresponding respectively to even and odd functions. In particular, the functions $f_{k,m,j}$ do not exist for $m\geq 2$. We can therefore drop the index $j$, and the resulting functions $f_{k,0}, f_{k,1}$ are  proportional to $h_{2k}, h_{2k+1}$, respectively, 
where
\begin{equation}\label{eq:hermite_functions}
    h_n(s)=\frac{(-1)^n}{\sqrt{2^n n!}}e^{s^2/2}\frac{\d^n}{\d s^n}(e^{-s^2})
\end{equation}
denote the usual Hermite functions.
\end{remark}
In Appendix \ref{app_symmetry}, the tangent subspace of the manifold $\Gg$ at $f=g_0$ is explicitly given by
\begin{equation}\label{eq:tangent_space_main_body}
    \begin{split}
        T_{g_0}\Gg &= \operatorname*{span}_{\mathbb C} \{g_0, x_1 g_0, \ldots , x_d g_0, |\boldsymbol{x}|^2 g_0\} \\ 
        &=\operatorname*{span}_{\mathbb C} \{f_{k, m,j}\ :\ (k=0, m\in\{0,1\}) \text{ or } (k=1, m=0)\}.
    \end{split}
\end{equation}
We conclude that $f \perp T_{g_0} \Gg$ if and only if $f$  admits an orthogonal expansion of the form
$$
f = {\sum_{k+m\geq 2}\sum_{j=1}^{N(d,m)}} a_{k,m,j} f_{k,m,j}.
$$
Plugging this expansion into~\eqref{eq_second_variation}, we obtain
\begin{align}\label{eq_orthogonal_decomp}
{\mathfrak{d}}_{q,r}''[g_0](f,f) ={\sum_{k+m\geq 2}\sum_{j=1}^{N(d,m)}} |a_{k,m, j}|^2 {\mathfrak{d}}_{q,r}''[g_0](f_{k,m,j},f_{k,m,j}).
\end{align}
In other words, the quadratic form ${\mathfrak{d}}_{q,r}''[g_0]$ is diagonal in the basis $\{f_{k,m,j}\}$. We proceed to verify this  for the last summand in~\eqref{eq_second_variation}, the proof for the remaining two summands being straightforward. Considering the real bilinear form 
\begin{equation}\label{eq:diagonalize_second_var}
    B(f, g):=\int_{-\pi}^\pi
    Lf(T)\,Lg(T)
    \d T,\quad \text{where}\quad Lf(T):=\Re\int_{\R^d}g_0^{r-1}(\boldsymbol{X})(e^{-iT(\Hc+d)/2}f)(\boldsymbol{X})\d \boldsymbol{X},
\end{equation}
we need to prove that $B(f_{k, m, j}, f_{k', m', j'})=0$ if $k\ne k'$ or $m\ne m'$ or $j\ne j'$. Now, since 
\begin{equation*}
    Lf_{k, m, j}(T)=\cos((2k+m)T)\int_0^\infty  L_k^{\nu+m}(\rho^2)\rho^{d+m-1}{e^{-\frac r2 \rho^2}}\, \d\rho \int_{\mathbb S^{d-1}} Y_{m, j}(\boldsymbol{\omega})\, \d\sigma(\boldsymbol{\omega}), 
\end{equation*}
we have that $Lf_{k, m, j}=0$ unless $m=0$, which then forces $j=1$. In that case, by orthogonality of the cosines on $(-\pi, \pi)$, we have $B(f_{k,0,1}, f_{k',0,1})=0$ unless $2k=2k'$, that is, $k=k'$. 

\noindent\textbf{Notational convention}. As shown by~\eqref{eq_orthogonal_decomp}, the degeneracy index $j$ of the spherical harmonics plays no role in the analysis. We will systematically omit it from now on. 

\begin{remark}\label{rem:tensor_hermite_no_good}
    In dimensions $d>1$, a commonly used basis of eigenfunctions for the number operator $-(\Hc + d)/2$ is obtained by taking tensor products of the Hermite functions~\eqref{eq:hermite_functions}. This basis is not suitable for our purposes because it does not diagonalize ${\mathfrak{d}}_{q,r}''[g_0]$.
\end{remark}

\begin{definition}\label{def:bottom_of_spectrum}
    Given a  Schrödinger-admissible pair $(q, r)$,  let 
    \begin{equation}\label{eq:bottom_of_spectrum}
        \mathfrak{S}_{q,r}(d):=\inf_{k+m\ge2} \frac{\mathfrak{d}''_{q, r}[g_0](f_{k,m}, f_{k,m} )}{q\lVert f_{k, m}\rVert_{L^2(\R^d)}^2},
    \end{equation}
    to which we refer as  the \emph{bottom of the spectrum} of $\mathfrak{d}''_{q, r}[g_0]$.        
\end{definition}
\begin{remark}
In view of Lemma~\ref{lem:deficit_dictionary}, the spectral gap $\Lambda_{q, r}(d)$ defined in \eqref{eq_SpecGap} is related to $\mathfrak{S}_{q, r}(d)$ as follows:
\begin{equation}\label{eq:spectral_gap_dictionary}
    \Lambda_{q,r}(d)
    =
    \pi^{\frac{2-q}{q}}
    \left(\frac{2\pi}{r}\right)^{\frac{d(2-q)}{2r}}
    \mathfrak{S}_{q,r}(d).
\end{equation}
\end{remark}
The quantity $\mathfrak{S}_{q, r}(d)$ determines the local maximizer nature of gaussians. Indeed, given $f\perp T_{g_0}\mathfrak{G}$, the following expansion holds:
\begin{equation}\label{eq:taylor_second_order}
    \mathfrak{d}_{q,r}[g_0+\varepsilon f]=\mathfrak{d}''_{q,r}[g_0](f,f)\frac{\varepsilon^2}{2} + o(\varepsilon^2). 
\end{equation}
As such, the local behaviour of $\mathfrak{d}_{q, r}$ at $g_0$ is determined by the sign of $\mathfrak{d}''_{q,r}[g_0]$, which coincides with the sign of $\mathfrak{S}_{q,r}(d)$.

\begin{proposition}\label{prop_modes}
Let $(q,r)$ be a Schrödinger-admissible pair. For every $(k,m) \in \Z^2_{\geq 0}$ with $k+m\geq 2$, we have
\begin{equation}\label{eq_prop_modes}
\frac{\pi^{\nu}}{\|g_0\|_{L^r(\R^d)}^q}
    \frac{{\mathfrak{d}}_{q,r}''[g_0](f_{k,m},f_{k,m})}
         {q\|f_{k,m}\|_{L^2(\R^d)}^2}
    =
    1-p_{k,m}(\tfrac2r).
\end{equation}
Here, $\nu=\frac d2-1$ and
\begin{equation}\label{eq_pkm_r}
    \begin{aligned}
    p_{k,m}(\lambda)
    :=&\;
    \lambda^{m-1}
    \sum_{j=0}^k
    \binom{k}{j}\binom{k+\nu+m}{k-j}
    \lambda^{2j}
    \left(1-\lambda\right)^{2k-2j} +
     \left(\frac{\nu+2}{\nu+1}-\frac1{\lambda}\right)
    \left(1-\lambda\right)^{2k-1}
    \binom{k+\nu}{k}\,\boldsymbol{\delta}_{m,0},
    \end{aligned}
\end{equation}
    where   $\boldsymbol{\delta}_{m,0}$ denotes the Kronecker delta.
\end{proposition}
\begin{proof}
Set
$\alpha=\nu+m$ and $\lambda=\frac2r$, and recall $\lVert g_0 \rVert_{L^r(\R^d)}^r = (\tfrac{2\pi}r)^{\frac d2}$. Direct computations reveal that
\begin{equation}\label{eq:good_bad_ugly}
        \begin{split}
            \frac{{\mathfrak{d}}_{q,r}''[g_0](f_{k, m}, f_{k, m})}{q {\lVert g_0 \rVert_{L^r}^q}}
           &= \frac{\pi}{\lVert g_0 \rVert_{L^2}^2} \|f_{k,m}\|_{L^2}^2   - \frac{r}{4}\lVert g_0\rVert_{L^r}^{-r} 2\pi \int_{0}^\infty L_k^{\nu+m}(\rho^2)^2 \rho^{2\nu+2m+1} e^{-\frac r 2\rho^2}\, \d\rho \\
            &\quad -\frac{\pi(q-r)}{2}\lVert g_0 \rVert_{L^r}^{-2r}  \left( \int_{0}^\infty  L_k^{\nu+m}(\rho^2)\rho^{2\nu+m+1} e^{-\frac r 2 \rho^2}\, \d\rho\right)^2\left( \int_{\mathbb S^{d-1}} Y_m\, \d\sigma\right)^2 \\
            & =  \pi^{-\nu}   \|f_{k,m}\|_{L^2}^2 - \frac12 \pi^{-\nu}  \la^{m-1} \int_{0}^\infty L_k^{\al}( \la s)^2 s^{\al} e^{-s}\, \d s \\
            & \quad -\frac{(q-r)\pi^{-\nu}}{4\nu!}   \left( \int_{0}^\infty  L_k^{\nu}(\la s)s^\nu e^{-s}\, \d s\right)^2\boldsymbol{\delta}_{m,0},      
        \end{split}
    \end{equation}
    where  we applied the change of variables $s=\rho^2/\la$.
   It follows from the dilation formula  \cite[Eq. (2.7)]{AIK} that
    \begin{align}\label{id:multiformulaLag}
        \frac{L_k^\al(\la t)}{L_k^\al(0)} = \sum_{j=0}^k \binom k j \la^j (1-\la)^{k-j} \frac{L_j^\al(t)}{L_j^\al(0)}.
    \end{align}
Since $L_k^\al(0)=\binom{k+\al} k$, this implies that 
    $\int_{0}^\infty  L_k^{\nu}(\la s)s^\nu e^{-s}\, \d s = \nu! (1-\la)^k L_k^\nu(0)$, and moreover
    $$
    \int_0^\infty L_k^{\al}(\la s)^2 s^{\al} e^{-s}\, \d s = \al!L_k^\al(0)\sum_{j=0}^k \binom k j \binom {k+\al} {k-j}\la^{2j}(1-\la)^{2k-2j}.
    $$
    Using that $\|f_{k,m}\|_{L^2}^2= \frac{\al!}2 L_{k}^\al(0) $ and that $\frac{q-r}2 = 
    \frac{\frac{\nu+2}{\nu+1}-\frac1{\la}}{1-\la}$, we finally obtain
    \begin{align}\label{eq:normalized_second_variation} 
         &\frac{\pi^{\nu}}{\lVert g_0 \rVert_{L^r}^q}\frac{{\mathfrak{d}}_{q,r}''[g_0](f_{k, m}, f_{k, m})}{q \|f_{k,m}\|_{L^2}^2}   \\\notag
         & = 1  -  {\la^{m-1}} \sum_{j=0}^k \binom k j \binom {k+\al} {k-j}{\la^{2j}}(1-\la)^{2k-2j} -\frac{\frac{\nu+2}{\nu+1}-\frac1{\la}}{1-\la} (1-\la)^{2k}L_k^\nu(0)\boldsymbol{\delta}_{m,0},
       \end{align}
       which concludes the proof of the proposition.
\end{proof}

\section{Local analysis close to gaussians: Proofs of Theorems~\ref{thm_main} and~\ref{thm_gap}}\label{sec_proofs}

Rather than working with the spatial dimension $d$ and the integrability parameter $r$, it will be convenient to switch to
\[
\nu=\frac d2-1,
\qquad
\lambda=\frac2r.
\]
Note that $\la \in (0,1)$. For $\nu\ge\frac12$, corresponding to $d\ge 3$, we additionally have $\la\in [\tfrac{\nu}{\nu+1},1)$, corresponding to the endpoint condition on $r$. 

In order to prove both Theorems~\ref{thm_main} and ~\ref{thm_gap}, we will study the quantity 
\begin{equation}\label{eq:inf_pkm}
\widetilde{\mathfrak{S}}_{q,r}:=  \inf_{k+m\geq 2} (1-p_{k,m}(\la)),
\end{equation}
which, as observed in Proposition~\ref{prop_modes}, coincides with the bottom of the spectrum of $\mathfrak{d}_{q, r}''[g_0]$ up to a strictly positive multiplicative constant. For the parameter range in Theorem~\ref{thm_main} (b), we will prove that $\widetilde{\mathfrak{S}}_{q,r}<0$.
Consequently, gaussians are not local maximizers for those cases, namely:\footnote{
We postpone the analysis of the case $(q,r,d)=(4,\infty,1)$ to \S\ref{sec_d1ri}.}
\begin{equation}\label{eq:gaussians_not_maximizers_nu}
\begin{array}{cc}
    \nu=-\frac12,& \rho_1<\frac2\lambda\le \infty, \\
    \nu=0, & \rho_2<\frac2\lambda<\infty, \\
    \nu=\frac12, & \rho_3<\frac2\lambda \le 6,\\
    \nu=1, & \rho_4<\frac2\lambda \le 4.
\end{array}
\end{equation}
We will also verify that $\widetilde{\mathfrak{S}}_{q,r}\ge 0$ in the remaining cases. More precisely, $\widetilde{\mathfrak{S}}_{q,r}= 0$ only if $d\in\{1,2,3,4,5\}$ and $r=\rho_d$, i.e., $\nu\in\{-\frac12, 0, \frac12, 1, \frac32\}$ and $\lambda=\frac2{\rho_d}$. In these exceptional cases, as observed in Remark \ref{rem_postmainthm} (c), we do not know whether gaussians are local maximizers. Finally, in all remaining cases, ${\widetilde{\mathfrak{S}}}_{q,r}>0$, which leads to Theorem~\ref{thm_main} (a). Along the way, we will compute the quantity ${\widetilde{\mathfrak{S}}}_{q,r}$ exactly, leading to the proof of Theorem~\ref{thm_gap}. To facilitate the verification of some upcoming algebraic identities  and the nonvanishing of certain polynomials, we include the relevant PARI/GP code \cite{gp} in Appendix~\ref{app_code}.

Given $k\geq 0$ and $\alpha \geq -1/2$, we define the polynomial
\begin{equation}\label{eq_Fka}
    F_{k,\alpha}(\lambda)
    :=
    \sum_{j=0}^k
    \binom{k}{j}\binom{k+\alpha}{k-j}
    \lambda^{2j}(1-\lambda)^{2k-2j}.
\end{equation}
From Proposition~\ref{prop_modes}, we have that 
\[
p_{k,m}(\lambda)=\lambda^{m-1}F_{k, \nu+m}(\lambda)+\left(\frac{\nu+2}{\nu+1}-\frac1\lambda\right)
    \left(1-\lambda\right)^{2k-1}
    \binom{k+\nu}{k}\boldsymbol{\delta}_{m,0}.
    \]    
We reduce the analysis of~\eqref{eq:inf_pkm} to the study of $1-p_{k,0}$ and $1-p_{k,1}$, via the following monotonicity lemma.\footnote{We remark that the proof of Lemma \ref{lem_master}  yields an alternative argument for the radialization result contained in~\cite[Proof of Theorem 1.2, Step 2]{GN22}.}
\begin{lemma}\label{lem_master}
The following monotonicity statements hold.
\begin{enumerate}
\item[(i)] Let $0<\la<1$, $k\geq 0$,  and $\al \geq 0$. Then
\[
\lambda^{\alpha+1} F_{k,\alpha+1}(\lambda) < \lambda^{\alpha} F_{k,\alpha}(\lambda).
\]

\item[(ii)] Let $k\geq 0$. For $\alpha\geq 0$, let
$
\frac{\alpha}{\alpha+2} < \lambda < 1,
$
while for $-\frac12 \leq \al<0$, let $
\frac{\al+1}{\al+3} \leq \lambda < 1.
$
Then
$$
F_{k+1,\alpha}(\lambda) < F_{k,\alpha}(\lambda).
$$
Moreover, if $\alpha>0$ and $\lambda=\frac{\alpha}{\alpha+2}$, then 
$F_{1,\alpha}(\lambda)=F_{0,\alpha}(\lambda)$, while for $k\geq 1$ we have
\[
F_{k+1,\alpha}(\lambda)<F_{k,\alpha}(\lambda).
\]
\end{enumerate}
\end{lemma}

\begin{proof}
For part (i), write
\[
S(\alpha):=\lambda^\alpha F_{k,\alpha}(\lambda).
\]
Set $Y=(\frac{1-\lambda}{\lambda})^2$ and $n=k+\alpha$. Since
$\binom{k}{j}=\binom{k}{k-j}$,
\[
S(\alpha)=\lambda^{2k+\alpha}\sum_{r=0}^k\binom{n}{r}\binom{k}{r}Y^r.
\]
Thus $S(\alpha+1)<S(\alpha)$ can be equivalently rewritten as
\begin{equation}\label{eq_ToTest}
\sum_{r=1}^k\binom{n}{r-1}\binom{k}{r}Y^r
<
\sqrt Y\sum_{r=0}^k\binom{n}{r}\binom{k}{r}Y^r.
\end{equation}
Introduce the symbols
\[
u_r:=\sqrt{\binom{n}{r}\binom{k}{r}Y^r},
\qquad
\beta_r:=
\sqrt{
\frac{\binom{n}{r-1}\binom{k}{r}}
     {\binom{n}{r}\binom{k}{r-1}}
}
=
\sqrt{\frac{k-r+1}{n-r+1}}.
\]
Since $n\geq k$, one has $\beta_r\leq 1$. Therefore
\[
\sum_{r=1}^k\binom{n}{r-1}\binom{k}{r}Y^r
=
\sqrt Y\sum_{r=1}^k u_{r-1}u_r\beta_r
\leq
\sqrt Y\sum_{r=1}^k u_{r-1}u_r
<
\sqrt Y\sum_{r=0}^k u_r^2,
\]
which proves \eqref{eq_ToTest} and thus part (i).

For part (ii), we first dispose of the case $\lambda=\frac12$.
By Vandermonde's identity,
\[
F_{k,\alpha}\left(\tfrac12\right)
=4^{-k}\binom{2k+\alpha}{k},
\]
and therefore
\[
\frac{F_{k+1,\alpha}(\frac12)}{F_{k,\alpha}(\frac12)}
=
\frac{(2k+\alpha+2)(2k+\alpha+1)}
     {4(k+1)(k+\alpha+1)}.
\]
The denominator minus the numerator equals
\[
(2-\alpha)(\alpha+1)+2k.
\]
If $\lambda=\frac12$ lies in the open range of part~(ii), then
$\alpha<2$, and this quantity is strictly positive. If instead
$\lambda=\frac{\alpha}{\alpha+2}=\frac12$, then $\alpha=2$, and the
same quantity equals $2k$, giving equality for $k=0$ and strict
inequality for every $k\geq1$. Thus all the required conclusions hold
when $\lambda=\frac12$, and we may henceforth assume
$\lambda\neq\frac12$.
Under this assumption, set $X=\frac{\la^2+(1-\la)^2}{1-2\la}$, and note that $\frac{X+1}{2}=\frac{(1-\la)^2}{1-2\la}$ and $\frac{X-1}{2}=\frac{\la^2}{1-2\la}$. Consequently,
$$
F_{k,\al}(\la) = (1-2\la)^k \sum_{j=0}^k \binom k j \binom {k+\al} {k-j}\left(\frac{X-1}{2}\right)^j \left(\frac{X+1}{2}\right)^{k-j} = (1-2\la)^k P_k^{\al,0}(X),
$$
where $P_k^{\al,0}(X)$ is the Jacobi polynomial. Therefore the three-term recurrence formula 
\[
F_{k,\al}(\la) = (A(\al,k)X+B(\al,k) )F_{k-1,\al}(\la) - C(\al,k)F_{k-2,\al}(\la)
\]
holds, see~\cite[\S18.9]{NIST:DLMF},
where the coefficients are given by
\begin{align*}
    A(\al,k) & :=\frac{(1-2\la)(2k+\al)(2k+\al-1)}{2k(k+\al)}, \\
B(\al,k)& :=\frac{(1-2\la)\al^2(2k+\al-1)}{2k(k+\al)(2k+\al-2)}, \\
C(\al,k)& :=\frac{(1-2\la)^2(2k+\al)(k-1)(k+\al-1)}{k(k+\al)(2k+\al-2)}.
\end{align*}
Observe  that $C(\al,k)>0$ for $k\geq 2$ and $\al\geq -1/2$. Assuming  $k\geq 2$ and $F_{k-2,\al}(\la) > F_{k-1,\al}(\la)$, we then see that 
$F_{k,\al}(\la)  < (A(\al,k)X+B(\al,k)-C(\al,k))F_{k-1,\al}(\la)$. We now claim that $A(\al,k)X+B(\al,k)-C(\al,k)\leq 1$, which will lead to the proof of part (ii) via induction. A direct computation reveals that
\[
A(\al,k)X+B(\al,k)-C(\al,k) = 1
-
(1-\lambda)\frac{
(\al + 2k)(\al^2 + \al + 2k - 2)\la -\al^2(\al+2k-1)
}{
k\,(\alpha+2k-2)\,(\alpha+k)
}.
\]
Hence, to prove our claim it will suffice to show that 
\[
(\al + 2k)(\al^2 + \al + 2k - 2)\la -\al^2(\al+2k-1) \geq 0.
\]
This function is increasing in $\la$ because $\al^2 + \al + 2k - 2\geq \al^2+\al+2\geq \frac74$ for $\al \geq -\frac12$. If $\al\geq 0$, then we can simply set $\la=\frac{\al}{\al+2}$, and the calculation yields $\frac{4k(k - 1)\al}{\al + 2}\geq 0$. If $-\frac12\leq\al<0$, then we set $\la = \frac{\al+1}{\al+3}$ to obtain
$(2\al^2 + (4k^2 - 2)(\al+1/2) + (2k^2+1 - 4k))/(\al + 3)$. In turn, this is an increasing function of $k \geq 2$, so we can set $k=2$ to obtain $(2\al^2 + 14\al + 8)/(\al + 3)>0$, which establishes  the claim. 

Away from the positive-$\alpha$ endpoint considered below, the induction
starts from
\[
F_{0,\alpha}(\lambda)-F_{1,\alpha}(\lambda)
=(1-\lambda)((\alpha+2)\lambda-\alpha)>0,
\]
and the result follows.
It remains to consider the endpoint
$\lambda=\frac{\alpha}{\alpha+2}$, with $\alpha>0$ and
$\alpha\neq2$. Here
\[
F_{0,\alpha}(\lambda)=F_{1,\alpha}(\lambda)=1,
\qquad
F_{1,\alpha}(\lambda)-F_{2,\alpha}(\lambda)
=\frac{8\alpha}{(\alpha+2)^4}>0.
\]
Moreover, the computation above gives, for $k\geq2$,
\[
(\alpha+2k)(\alpha^2+\alpha+2k-2)\lambda
-\alpha^2(\alpha+2k-1)
=\frac{4k(k-1)\alpha}{\alpha+2}>0.
\]
The same induction therefore yields
$F_{k+1,\alpha}(\lambda)<F_{k,\alpha}(\lambda)$ for every $k\geq1$.
\end{proof}

Recall that $\la\in [\frac{\nu}{\nu+1},1)$, corresponding to the endpoint condition on $r$. Now that the  monotonicity prerequisites have been settled, we observe that $p_{k,m}(\la) = \la^{-1-\nu} \la^{\nu+m}F_{k,\nu+m}(\la)$ if $m\geq 1$. For $d\geq 2$, i.e., $\nu\geq 0$, we can directly apply Lemma \ref{lem_master}   (i) to obtain 
$$
\inf_{k+m\geq 2} (1-p_{k,m}(\la)) = \min\bigg(  1-p_{1,1}(\la), 1-p_{0,2}(\la) ,\inf_{k\geq 2}( 1-p_{k,0}(\la)) , \inf_{k\geq 2}(1-p_{k,1}(\la)) \bigg),
$$
for all $\la\in (0,1)$. For $d=1$, i.e., $\nu=-\frac12$,  only  two harmonics $m\in \{0,1\}$ exist, and  instead  
$$
\inf_{k+m\geq 2} (1-p_{k,m}(\la)) = \min\bigg( 1-p_{1,1}(\la), \inf_{k\geq 2}( 1-p_{k,0}(\la)) , \inf_{k\geq 2}(1-p_{k,1}(\la))\bigg).
$$
The term $1-p_{1,1}(\lambda)$  simplifies to
$$
\frac{1-p_{1,1}(\la)}{1-\la}= (\nu + 3)\la - (\nu + 1),
$$
hence $\text{sign}(1-p_{1,1}(\la)) = \text{sign}(\la-\tfrac{\nu+1}{\nu+3})$.  We can thus restrict the analysis to the range $\frac{\nu+1}{\nu+3}\leq  \la <1$. Since $p_{k,1}(\la)=F_{k,\nu+1}(\la)$, we can apply  Lemma \ref{lem_master} (ii) for $\max\left(\frac{\nu}{\nu+1},\frac{\nu+1}{\nu+3}\right)\leq  \la <1$ to conclude that for $d\geq 2$ we have
$$
\inf_{k+m\geq 2} (1-p_{k,m}(\la)) = \min\bigg( 1-p_{1,1}(\la), 1-p_{0,2}(\la) ,  \inf_{k\geq 2}( 1-p_{k,0}(\la)  )\bigg),
$$
while for $d=1$ we have
$$
\inf_{k+m\geq 2} (1-p_{k,m}(\la)) = \min\bigg( 1-p_{1,1}(\la),\inf_{k\geq 2}( 1-p_{k,0}(\la)  )\bigg).
$$
Further note that, for $k \geq 2$,
$$
p_{k,0}(\la)=\la^{-1}\bigg(F_{k,\nu}(\la) + (\la\tfrac{\nu+2}{\nu+1}-1 )(1-\la)^{2k-1}L_k^\nu(0)\bigg).
$$
Observing also that 
$$
\frac{(1-\la)^{2k+1}L_{k+1}^\nu(0)}{(1-\la)^{2k-1}{L_{k}^\nu(0)} } =(1-\la)^2 \frac{\nu+k+1}{k+1} <1
$$
in the range $\max\left(\frac{\nu}{\nu+1},\frac{\nu+1}{\nu+3}\right)\leq  \la <1$, we conclude that $p_{k,0}(\la)$ is a decreasing sequence for $k\geq 2$  when $\frac{\nu+1}{\nu+2}< \la<1$. Thus  $$\inf_{k\geq 2}( 1-p_{k,0}(\la) ) = 1-p_{2,0}(\la)$$ in this range. We next claim that, for every $d\geq 1$ and $\max\left(\frac{\nu}{\nu+1},\frac{\nu+1}{\nu+3}\right)\leq  \la   \leq  \frac{\nu+1}{\nu+2}$, it holds 
\begin{equation}\label{eq_claim}
\sup_{k\geq 4}p_{k,0}(\la)  < p_{1,1}(\la)  \quad \text{and} \quad p_{3,0}(\la) < p_{2,0}(\la) .
\end{equation}
We establish this claim at the end of the proof. We also have
$$
\frac{p_{1,1}(\la) - p_{2,0}(\la)}{{1-\la}} =
\begin{cases}
\frac{13 \la^2}{4}-3 \la+\frac{3}{4}, & d=1,\\
(2\la-1)^2, & d=2.
\end{cases}
$$
The second  polynomial is visibly nonnegative everywhere, while the first one also has that property since  its discriminant equals $-\frac34$.  These facts together imply, in the range $\max\left(\frac{\nu}{\nu+1},\frac{\nu+1}{\nu+3}\right)\leq  \la <1$, 
$$
(d\geq 3) \qquad \inf_{k+m\geq 2} (1-p_{k,m}(\la)) = \min\bigg( 1-p_{1,1}(\la), 1-p_{0,2}(\la) ,   1-p_{2,0}(\la)  \bigg),
$$
$$
(d=2) \qquad \inf_{k+m\geq 2} (1-p_{k,m}(\la)) = \min\bigg( 1-p_{1,1}(\la),  1-p_{0,2}(\la)  \bigg),
$$
and 
$$
(d=1) \qquad \inf_{k+m\geq 2} (1-p_{k,m}(\la)) = 1-p_{1,1}(\la).
$$

Since $p_{0,2}(\la)=\la <1$, it does not interfere with any sign analysis of these minima. Let $d\in\{1,2\}$ (i.e., $\nu\in \{-\frac12,0\}$). Since $\max\left(\frac{\nu}{\nu+1},0\right)\leq 0 < \frac{\nu+1}{\nu+3}$ and $\text{sign}(1-p_{1,1}(\la)) = \text{sign}(\la-\tfrac{\nu+1}{\nu+3})$, we conclude that gaussians are not local maximizers for $\rho_d<r<\infty$ with $\rho_1=10$ and $\rho_2=6$, but are stable local maximizers for $2<r<\rho_d$. Now let $d\in \{3,4,5\}$ (i.e., $\nu\in\{\frac12,1,\frac32\}$). Since
$
\frac{1-p_{2,0}(\la)}{1-\la} = \left(\frac{3}{2}\nu
 + 4\right) \la^2
 - \left(\nu
 + 1\right) \la
 - \frac{1}{2}\nu,
$
we obtain 
\begin{align}\label{eq_factP20}
\frac{1-p_{2,0}(\la)}{1-\la} = \begin{cases}  \frac{19^2}{76} \left(\la-\frac{3+2 \sqrt{7}}{19}\right) \left(\la-\frac{3-2 \sqrt{7}}{19}\right),  & d=3, \\ 
 \frac{11^2}{22} \left(\la-\frac{2+\sqrt{15}}{11}\right) \left(\la-\frac{2-\sqrt{15}}{11}\right), & d=4, \\ 
\frac{25}{4} \left(\la-\frac{3}{5}\right) \left(\la+\frac{1}{5}\right), & d=5.
\end{cases}
\end{align} 
For $d=3,4$, we respectively have
$$
\frac{3+2 \sqrt{7}}{19} \approx 0.436 > \max\left(\frac{\nu}{\nu+1},\frac{\nu+1}{\nu+3}\right) = \frac{3}{7} \approx 0.429
$$
and
$$
\frac{2+\sqrt{15}}{11}  \approx 0.533 > \max\left(\frac{\nu}{\nu+1},\frac{\nu+1}{\nu+3}\right) = \frac{1}{2}.
$$
This shows that gaussians are not local maximizers for $\rho_d<r \leq \frac{2d}{d-2}$ for $d\in\{3,4\}$ with $\rho_3=4\sqrt{7}-6$ and $\rho_4 = 2\sqrt{15}-4$, but that gaussians are stable local maximizers for $2<r<\rho_d$. For $d=5$, we have $\frac{3}{5}= \frac{\nu}{\nu+1} > \frac{\nu+1}{\nu+3}$, hence the analysis is inconclusive at $r=\rho_5=\frac{10}{3}$, but gaussians are still stable local maximizers for $2<r<\rho_5$.
Finally, suppose that $d\geq6$, so that $\nu\geq2$. Since
$    \frac{\nu}{\nu+1}>\frac{\nu+1}{\nu+3}$,
we have $1-p_{1,1}(\lambda)>0$ throughout the admissible range.
Moreover, setting
\[
    Q_\nu(\lambda):=
    \frac{1-p_{2,0}(\lambda)}{1-\lambda},
\]
we have
\[
    Q_\nu\left(\frac{\nu}{\nu+1}\right)
    =\frac{\nu(2\nu-3)}{2(\nu+1)^2}>0,
\]
and $Q_\nu$ is increasing for
$\lambda\geq\nu/(\nu+1)$. Hence all three quantities in the
minimum above are strictly positive throughout the admissible range.

It remains to establish the claim made in \eqref{eq_claim}.
For $k\geq 4$, we can use Lemma \ref{lem_master}  (ii) to obtain
$$
p_{k,0}(\la)=\la^{-1}\bigg(F_{k,\nu}(\la) + (\la\tfrac{\nu+2}{\nu+1}-1 )(1-\la)^{2k-1}L_k^\nu(0)\bigg) \leq \la^{-1}F_{4,\nu}(\la).
$$
It then suffices  to show that,  for every $d\geq 1$ and $\max\left(\frac{\nu}{\nu+1},\frac{\nu+1}{\nu+3}\right)\leq  \la  <\infty$, we have
$$
\la^{-1}F_{4,\nu}(\la) < p_{1,1}(\la)  \quad \text{and} \quad p_{3,0}(\la) < p_{2,0}(\la) .
$$
Define the polynomials
$$
g_1(\la):=\frac{\la p_{1,1}(\la)-F_{4,\nu}(\la)}{1-\la} \quad \text{and} \quad g_2(\la):=\frac{p_{2,0}(\la)-p_{3,0}(\la)}{1-\la}.
$$
For $g_1$ when $1\leq d\leq 4$, and for $g_2$ when $1\leq d\leq 8$, Sturm's method with exact rational arithmetic reveals that these polynomials have no zeros in the range $\max\left(\frac{\nu}{\nu+1},\frac{\nu+1}{\nu+3}\right)\leq\la<\infty$. Since $g_1(1)=5$ and $g_2(1)=2$, they are strictly positive in this range. In the remaining cases, namely, $5\leq d\leq 9$ for $g_1$ and $d=9$ for $g_2$, translating the left endpoint to zero yields polynomials with positive coefficients. The desired inequalities follow. For $d\geq 10$, that is, $\nu\geq 4$, we let $y=\la-\frac{\nu}{\nu+1}$
and write $g_1(\la)$ and $g_2(\la)$ as polynomials in the new variable $y\geq 0$. 
Write the coefficients of $g_1$ and $g_2$ as rational functions of $\nu$, i.e., $g_1=\sum_{j=0}^7 \frac{A_j(\nu)}{B_j(\nu)} y^j$ and $g_2=\sum_{k=0}^4 \frac{C_k(\nu)}{D_k(\nu)} y^k$. The denominators $B_j$ and $D_k$ are polynomials with nonnegative coefficients. It is a routine verification (e.g., with elementary computer assistance) that $A_j(\mu + 4)$ and $C_j(\mu+4)$ are polynomials in $\mu$ with nonnegative coefficients. Since $\nu=\mu+4\ge 4$ if and only if $\mu\ge 0$, this proves the desired claim.  
We note that the  claim is not true for $\nu\geq 3$, since the coefficient of $y^2$ in $g_2(\la)$ has a simple root at $\nu\approx 3.108$.

\begin{remark}\label{rem_dimKer}
    
At the threshold $r=\rho_d$, the preceding analysis also identifies the
kernel of the second variation on $(T_{g_0}\Gg)^\perp$. Indeed,
by \eqref{eq_orthogonal_decomp} and Proposition \ref{prop_modes}, the eigenvalues are, up to a strictly
positive multiplicative constant,
\[
    1-p_{k,m}\!\left(\frac{2}{\rho_d}\right),
    \qquad k+m\geq 2.
\]
For $d\in\{1,2\}$, the only vanishing eigenvalue is the one corresponding
to $(k,m)=(1,1)$, since
\[
    \frac{1-p_{1,1}(\lambda)}{1-\lambda}
    =(\nu+3)\lambda-(\nu+1)
\]
vanishes at $\lambda=2/\rho_d$. Hence the kernel is
$\operatorname{span}_{\mathbb C}\{f_{1,1}\}$ when $d=1$, and the full
degree-one spherical-harmonic eigenspace when $d=2$, which has complex
dimension $N(2,1)=2$. For $d\in\{3,4,5\}$, the only vanishing eigenvalue
is the one corresponding to $(k,m)=(2,0)$, as follows from the
factorizations of $(1-p_{2,0}(\lambda))/(1-\lambda)$ in \eqref{eq_factP20} together with
the strict comparisons in \eqref{eq_claim}. Since this mode is radial, its
multiplicity is $N(d,0)=1$. Therefore the second variation is positive
semidefinite on $(T_{g_0}\Gg)^\perp$, and its kernel has complex
dimension $1$, except when $d=2$, in which case it has complex dimension
$2$.

\end{remark}

\subsection{\texorpdfstring{{The case $d=1$, $r=\infty$}}{The d=1, r=infinity case}}\label{sec_d1ri}

When $d=1$, the pair $(q,r)=(4,\infty)$ is Schrödinger-admissible, and, as stated in Theorem~\ref{thm_main}~(b), gaussians are not local maximizers. However, the proof requires an {\it ad hoc} argument because Lemma~\ref{lem_second_variation} does not apply when $r=\infty$, and neither does  the method of the previous section. We state and prove such an argument as an independent proposition. Recall the definition~\eqref{eq:hermite_functions} of the Hermite functions $\{h_n\}$. In particular, the standard one-dimensional gaussian satisfies  $g_0(x)=e^{-x^2/2}=h_0(x)$. We have $\lVert h_n\rVert_{L^2(\R)}^4=\pi$, and so $\tilde h_n(x):=\pi^{-1/4}h_n(x)$ denote the $L^2$-normalized Hermite functions.
\begin{proposition}\label{prop:infty_case}
    Given  $n\in 2\mathbb Z_{\geq 0}+1$, let $g_\varepsilon:=\frac{\tilde{h}_0+\varepsilon\tilde{h}_n}{\sqrt{1+\varepsilon^2}}$, so that $\lVert g_\varepsilon\rVert_{L^2(\R)}=1$ for all $\varepsilon\in\mathbb R$. Then
    \begin{equation}\label{eq:taylor_exp_infty}
        \mathfrak{d}_{4,\infty}[g_\varepsilon]=-\frac12C_n\varepsilon^2+o(\varepsilon^2),
    \end{equation}
    with 
    \begin{equation}\label{eq:const_infty}
        C_n=\frac{8n^2}{2^nn!}\left(\frac{(n-1)!}{\left(\frac{n-1}{2}\right)!}\right)^2-4.
    \end{equation}
    Moreover, $C_n>0$ for all odd $n\ge 3$. Thus the gaussian $g_0$ is not a local maximizer.
\end{proposition}
\begin{proof}
    Observing $\lVert g_0\rVert_{L^\infty}=1$, by~\eqref{eq:deficit_post_lens}  we have
    \begin{equation}\label{eq:deficit_infty_after_lens}
        \begin{split}
            2\mathfrak{d}_{4,\infty}[g_\varepsilon]&=\frac{\lVert e^{-iT(\Hc+d)/2}g_0 \rVert_{L^4_TL^\infty_{{X}}((-\pi, \pi)\times \mathbb R)}^4}{\lVert g_0\rVert_{L^2}^4} - \lVert e^{-iT(\Hc+d)/2} g_\varepsilon \rVert_{L^4_TL^\infty_{{X}}((-\pi, \pi)\times \mathbb R)}^4 \\
            &=2- {\int_{-\pi}^\pi \sup_{X\in\R}\left\lvert \frac{\tilde{h}_0(X)+\varepsilon e^{iTn}\tilde{h}_n(X)}{\sqrt{1+\varepsilon^2}}\right\rvert^4\, \d T}=:2-F(\varepsilon).
        \end{split}
    \end{equation}
    Below we show that $F$ is smooth in a neighborhood of $\varepsilon=0$ and expand it to second order. Note that $F(-\varepsilon)=F(\varepsilon)$, as can be established via the change of variables $T\mapsto T-\frac{\pi}n$, and so $F(\varepsilon)=2+O(\varepsilon^2)$.\footnote{This argument can  be used to produce a simpler proof that $g_0$ is a critical point for all Schrödinger-admissible pairs of exponents, as originally established by Christ--Quilodrán \cite{CQ14}.} 
    
   For the purpose of computing the second order expansion, let
\begin{equation*}
    G_{\varepsilon,T}(X)
    :=
    \left\lvert
        \tilde h_0(X)+\varepsilon e^{inT}\tilde h_n(X)
    \right\rvert^2.
\end{equation*}
Clearly $G_{\varepsilon,T}\to\tilde h_0^2$ uniformly in $X$ and $T$, as $\varepsilon\to 0^+,$ and
$\tilde h_0^2(X)=\pi^{-1/2}e^{-X^2}$. 
In particular, the limit has a unique global maximum at $X=0$ with strictly negative second derivative. Uniform convergence and the strict global maximum of $\tilde h_0^2$ imply
that every global maximizer converges to $0$, as $\varepsilon\to0^+$, uniformly
in $T$. On the other hand, by the strict negativity of that second derivative, the implicit function theorem applied to the critical point
equation
\begin{equation*}
    \partial_XG_{\varepsilon,T}(X)=0
\end{equation*}
reveals that, for sufficiently small $\varepsilon$, there exists a unique
critical point near $0$, which is therefore the unique global maximizer.
This implicit function theorem argument also proves that $F$ depends smoothly on $\varepsilon$, as claimed above.    

By the preceding observations, we can compute the supremum in the definition of $F(\varepsilon)$ via a Taylor expansion around $X=0$. To do this, observe that
    \begin{equation}\label{eq:local_def_g}
        \tilde{h}_0(X)+\varepsilon e^{inT}\tilde{h}_n(X)=\pi^{-\frac14}\left(1-\frac{X^2}{2}+O(X^4)\right)+\varepsilon e^{inT}\left( \tilde{h}_n'(0)X+O(X^3)\right),
    \end{equation}
    so we can compute 
    \begin{equation}\label{eq:expand_gsquare}
        G_{\varepsilon, T}(X)=\pi^{-\frac12}(1-X^2)+2\pi^{-\frac14}\cos(nT)\tilde{h}_n'(0)\varepsilon X + O(\varepsilon^2X^2)+O(X^3),
    \end{equation}
    which attains its maximum at $X_\star=\pi^{1/4}\cos(nT)\tilde{h}_n'(0)\varepsilon+O(\varepsilon^2)$. We conclude that  
    \begin{equation}\label{eq:compute_sup_second order}
        \sup_{X\in\mathbb R}G_{\varepsilon, T}(X)^2 = \pi^{-1} + 2\pi^{-\frac12}\cos^2(nT)(\tilde{h}_n'(0))^2\varepsilon^2 + O(\varepsilon^3).
    \end{equation}
   Recalling the definition of $\tilde{h}_n$, we have that $\tilde{h}_n'(0)=\pi^{-1/4}(2^n n!)^{-1/2}H'_n(0)$ where $H_n$ is the Hermite polynomial of degree $n$. The known formula $H'_n(x)=2nH_{n-1}(x)$ and the expansion $$H_{n-1}(x)=(n-1)! \sum_{k=0}^{\lfloor \frac{n-1}2\rfloor} \frac{(-1)^k(2x)^{n-1-2k}}{k!(n-1-2k)!},$$ see ~\cite[(18.9.25),(18.5.13)]{NIST:DLMF}, together imply\footnote{Recall that $n$ is odd, hence $\frac{n-1}2$ is an integer.}
    \begin{equation}\label{eq:derivative_normalized_hermite}
        \tilde{h}_n'(0)=(-1)^\frac{n-1}{2}\frac{2n}{\pi^\frac14\sqrt{2^n n!}}\frac{(n-1)!}{\left(\frac{n-1}{2}\right)!}.
    \end{equation}
    Multiplying by $(1+\varepsilon^2)^{-2}=1-2\varepsilon^2+O(\varepsilon^4)$ and integrating in $T$, we conclude that
    \begin{equation}\label{eq:expansion_integral}
        F(\varepsilon)=2+C_n \varepsilon^2+ O(\varepsilon^3),
    \end{equation}
    as claimed. To finish the proof, it remains to show that $C_n>0$ for $n\ge 3$ odd. This is readily seen by writing $n=2m+1$ with $m\in\mathbb N$ and noting that
    \[C_{2m+1}=4\frac{2m+1}{4^m}\binom{2m}{m}-4=:4a_m-4.\]
    Now, $\frac{a_{m+1}}{a_m}=\frac{2m+3}{2m+2}>1$ and $a_1=\frac 32>1$, so $a_m>1$ for every $m\in\mathbb N$. The conclusion follows.
\end{proof}

\section{Effective stability}\label{sec:stability}

 Our proof of Theorem~\ref{thm:effective_stability} will follow a general blueprint which we state and prove in abstract terms, in the vein of~\cite{GN22}. The resulting abstract Theorem~\ref{thm:abstract_stability} will then be applied to prove Theorem~\ref{thm:effective_stability} at the end of the section. 
 
 Let $\mathcal H$ be a complex\footnote{It is not difficult to formulate a version of the results of this section for real Hilbert spaces, in a way similar to~\cite[Remark~2.2]{GN22}.} Hilbert space, with norm $\|\cdot\|:=\|\cdot\|_{\mathcal H}$ induced by its inner product. Given exponents $q,r>3$, consider a linear operator $S\colon \mathcal H\to L_t^q L_{\boldsymbol{x}}^r(\mathbb R^{1+d})$ with operator norm $\lVert S\rVert$, and define the associated 2-homogeneous deficit functional
\begin{equation}\label{eq:general_deficit_two}
    D[f]:=\lVert S\rVert^2\lVert f\rVert^2 - \lVert Sf\rVert_{L^qL^r}^2.
\end{equation}
Note that $D$ is thrice real differentiable at every $f$ such that $Sf$ is nonzero, since $q,r>3$. For $f,g\in\mathcal H$, we denote the first three Fréchet derivatives at $g$ in the direction of $f$ by
\begin{equation}\label{eq:abstract_derivatives}
    \begin{array}{ccc}
         D'[g](f):=\left.\frac{\partial}{\partial\varepsilon}\right|_{\varepsilon=0}D[g+\varepsilon f],
         & D''[g](f,f):=\left.\frac{\partial^2}{\partial\varepsilon^2}\right|_{\varepsilon=0}D[g+\varepsilon f],
         & D'''[g](f,f, f):=\left.\frac{\partial^3}{\partial\varepsilon^3}\right|_{\varepsilon=0}D[g+\varepsilon f].
    \end{array}
\end{equation}
Let $M$ denote the set of all \emph{maximizers}, i.e., 
\begin{equation}\label{eq:maximisers_in_general}
    M=\{f\in\mathcal H\ :\ D[f]=0\},
\end{equation}
and suppose there exists $h_0\in M\setminus\{0\}$. We require the existence of  a Lie group of Hilbert space isometries 
 \begin{equation}\label{eq:generic_lie_group}
    G:=\{\Gamma_{\boldsymbol{\theta}}\colon \mathcal H \to \mathcal H\ :\ \boldsymbol{\theta}\in \mathbb R^k\},
\end{equation}
such that $\lVert S\Gamma_{\boldsymbol{\theta}}f\rVert_{L^qL^r}$ is independent of  $\boldsymbol{\theta}$. We further assume that $G$ vanishes at infinity, in the sense that $\langle\Gamma_{\boldsymbol{\theta}}f, g\rangle_{\mathcal H}\to 0$ as $\lvert\boldsymbol\theta\rvert\to\infty$. These assumptions ensure that the orbit
\begin{equation}\label{eq:orbit_hzero}
    \mathcal{O}:= (\mathbb C^\times G).h_0=\{c\Gamma_{\boldsymbol{\theta}}h_0\ :\ c\in \mathbb C\setminus\{0\},\ \boldsymbol{\theta}\in \mathbb R^k\}
\end{equation}
is a finite-dimensional smooth manifold that admits a (possibly non-unique) metric projection whenever $d(f,\mathcal O)<\lVert f\rVert$; see the proof of Lemma~\ref{lem:third_estimate} below. Finally, let $\overline{\mathcal{O}}:=\mathcal{O}\cup\{0\}$.

This setup is quite general and applies, in particular, to all known cases of  Sobolev, Strichartz and Fourier extension  inequalities whose best constants are known. The following theorem deduces an effective stability estimate under two extra assumptions, a global one (i) and a local one (ii). We remark that these assumptions are quite strong, and establishing them in specific cases of interest may amount to a formidable task.
\begin{theorem}\label{thm:abstract_stability}
    Suppose that the operator $S$ satisfies the following properties.
    \begin{enumerate}\label{enum:extra_assumptions}
        \item[(i)] \underline{Existence of a flow converging to the maximizers.} There exists a family $(\Phi_s)_{s\ge 0}$ of (possibly nonlinear) norm-preserving operators on $\mathcal H$ such that $\Phi_0$ is the identity and, for every $f\in\mathcal H$, the map $s\mapsto\Phi_s(f)$ is continuous, while $D[\Phi_s(f)]$ is strictly decreasing as long as $\Phi_s(f)\notin\overline{\mathcal O}$ and vanishes whenever $\Phi_s(f)\in\overline{\mathcal O}$. Moreover,
        \begin{equation}\label{eq:Phis_to_zero}
            \begin{array}{cc}
                \displaystyle\lim_{s\to\infty}D[\Phi_s(f)]=0,
                &
                \displaystyle\lim_{s\to\infty}d(\Phi_s(f),\mathcal O)=0.
            \end{array}
        \end{equation}
        \item[(ii)] \underline{Positive spectral gap.} There exists $\lambda>0$ such that $\frac12D''[h_0](f_\bot, f_\bot)\ge \lambda \lVert f_\bot\rVert^2$, for every $f_\bot\in\mathcal H$ which is orthogonal to the tangent space $T_{h_0}\mathcal O$.   
    \end{enumerate}
    Then $M=\overline{\mathcal O}$ and, letting $\eta$ denote the explicit \emph{local radius} of $D$ around $\mathcal O$ (see  Lemma~\ref{lem:third_estimate} below), we have that
    \begin{equation}\label{eq:explicit_stability}
        D[f]
        \ge
        \frac{\lambda\eta^2}{3}d(f,M)^2,
    \end{equation}
    for every $f\in\mathcal H$.
\end{theorem}
The first conclusion of the theorem, namely that $M=\overline{\mathcal O}$, is  an immediate consequence of property~(i). The more interesting conclusion is that the constant $\lambda \eta^2/3$ in \eqref{eq:explicit_stability}, albeit possibly suboptimal, is explicit.

In order to prove Theorem~\ref{thm:abstract_stability}, we  first establish the following lemma, which is an improved version of~\cite[Theorem~2.1]{GN22}. The improvement lies in the fact that the local radius  $\eta$  is effective. Since this lemma is purely local, it does not rely on assumption (i) above.
\begin{lemma}[Effective Local Stability]\label{lem:third_estimate}
Let $K_{q,r}>0$ be such that, for every $g, h\in \mathcal H$,
\begin{equation}\label{eq:third_derivative_bound}
    \left|D'''[h](g,g,g)\right|
    \le
    K_{q,r}
    \frac{\lVert Sg\rVert_{L^qL^r}^3}
         {\lVert Sh\rVert_{L^qL^r}}
\end{equation}
whenever $Sh$ is nonzero, and define the \emph{local radius}
\begin{equation}\label{eq:eta_in_terms_of_third_der}
    \eta
    :=
    \min\left\{
        \frac{1}{\sqrt{5}},
        \frac{2\lambda}
        {\sqrt{K_{q,r}^2\lVert S\rVert^4+4\lambda^2}}
    \right\}.
\end{equation}
Then, for every $f\in\mathcal H$ such that $d(f,\mathcal{O})\le\eta\lVert f\rVert$, it holds 
\begin{equation}\label{eq:effective_local_taylor}
    D[f]\ge \frac{\lambda}{3}d(f,\mathcal{O})^2.
\end{equation}
An explicit choice of $K_{q,r}$ is
\begin{equation}\label{eq:K_explicit}
    \begin{split}
        K_{q,r}
        ={}&
        2\left(
            (r-2)(r+1)
            +3|q-r|(r-1)
            +|q-r||q-2r|
        \right)\\
        &+
        6(q-2)\left(r-1+|q-r|\right)
        +
        4(q-2)(q-1).
    \end{split}
\end{equation}
\end{lemma}

\begin{proof}
    The conclusion is immediate if $f=0$, so suppose that $f$ is nonzero.
    Since $\eta<1$, the strict inequality $d(f,\mathcal{O})\le\eta\lVert f\rVert<\lVert f\rVert$ holds.
    Now, as observed in~\cite[Theorem~2.1]{GN22}, this strict inequality implies that there exists at least one metric projection $P(f)$ of $f$ onto $\mathcal{O}$, i.e., $f=P(f)+f_\bot$ with $\lVert f_\bot\rVert=d(f, \mathcal{O})$. In particular $P(f)$ is nonzero. Thus, after applying an element of $G$ and a phase rotation, both of which leave the deficit and the distance unchanged, we may assume 
    \begin{equation}\label{eq:f_h_zero_fbot}
        f=ch_0+f_\bot,
        \qquad
        c>0,
        \qquad
        f_\bot\perp T_{h_0}\mathcal{O},
    \end{equation}
    with
    $
        \lVert f_\bot\rVert=d(f,\mathcal{O}).
    $
    Since $h_0\in T_{h_0}\mathcal{O}$, we also have $f_\bot\perp h_0$, and hence
    \begin{equation}\label{eq:c_in_terms_of_distance}
        c^2\lVert h_0\rVert^2
        =
        \lVert f\rVert^2-d(f,\mathcal{O})^2.
    \end{equation}
    
    By construction, $D[h_0]=0$. Actually, $h_0$ is a global minimizer of $D$ and so $D'[h_0]=0$. Taylor expansion to second order therefore yields
    \begin{equation}\label{eq:second_order_taylor}
        \begin{split}
            D[f]
            =
            c^2D\left[h_0+\frac{f_\bot}{c}\right]
            &=
            c^2\left(
                \frac{D''[h_0](f_\bot,f_\bot)}{2c^2}
                +
                R\left(\frac{f_\bot}{c}\right)
            \right) \\
            &\ge
            \lambda\lVert f_\bot\rVert^2
            +
            c^2R\left(\frac{f_\bot}{c}\right),
        \end{split}
    \end{equation}
    for some remainder  $R$. By the Lagrange form of the remainder, there exists $\theta\in(0,1)$ such that
    \begin{equation*}
        R(g)
        =
        \frac{1}{6}D'''[h_0+\theta g](g,g,g),
    \end{equation*}
    where we wrote $g=f_\bot/c$.
    By the definition of $K_{q,r}$, we may upper bound the remainder as follows:
    \begin{equation}\label{eq:first_bound}
        \begin{split}
            \lvert R(g)\rvert
            \le
            \frac{K_{q,r}}{6}
            \frac{\lVert Sg\rVert_{L^qL^r}^3}
                 {\lVert Sh_0+\theta Sg\rVert_{L^qL^r}}
            .
        \end{split}
    \end{equation}
    Using~\eqref{eq:c_in_terms_of_distance} and $d(f, \mathcal{O})\le \eta \lVert f\rVert$, the condition $\eta\le \frac1{\sqrt 5}$ implies that
    \begin{equation*}
        \frac{\lVert g\rVert}{\lVert h_0\rVert}
        =
        \frac{d(f,\mathcal{O})}
             {\sqrt{\lVert f\rVert^2-d(f,\mathcal{O})^2}}
        \le
        \frac{\eta}{\sqrt{1-\eta^2}}
        \le \frac12.
    \end{equation*}
    Consequently, for every $\theta\in[0,1]$,
    \begin{equation*}
        \begin{split}
            \lVert Sh_0+\theta Sg\rVert_{L^qL^r}
            &\ge
            \lVert Sh_0\rVert_{L^qL^r}
            -\theta\lVert Sg\rVert_{L^qL^r} \\
            &\ge
            \lVert S\rVert\bigl(\lVert h_0\rVert-\lVert g\rVert\bigr)
            \ge
            \frac12\lVert S\rVert\lVert h_0\rVert,
        \end{split}
    \end{equation*}
    since $\lVert Sh_0\rVert_{L^qL^r}=\lVert S\rVert\lVert h_0\rVert$. Plugging this back into~\eqref{eq:first_bound}, we obtain  $\lvert R(g)\rvert \le
            \frac{K_{q,r}\lVert S\rVert^2}{3}
            \frac{\lVert g\rVert^3}{\lVert h_0\rVert}$, and therefore~\eqref{eq:second_order_taylor} yields
    \begin{equation*}
        D[f]
        \ge
        \lambda\lVert f_\bot\rVert^2
        -
        \frac{K_{q,r}\lVert S\rVert^2}{3}
        \frac{\lVert f_\bot\rVert^3}
             {c\lVert h_0\rVert}.
    \end{equation*}
    Using~\eqref{eq:c_in_terms_of_distance} a second time together with $\lVert f_\bot\rVert=d(f,\mathcal{O})$, we infer the lower bound
    \begin{equation}\label{eq:bound_below_Df}
        D[f]
        \ge
        \left(
            \lambda
            -
            \frac{K_{q,r}\lVert S\rVert^2}{3}
            \frac{d(f,\mathcal{O})}
                 {\sqrt{\lVert f\rVert^2-d(f,\mathcal{O})^2}}
        \right)d(f,\mathcal{O})^2,
    \end{equation}
    which, again by $d(f,\mathcal{O})\le\eta\lVert f\rVert$, yields via~\eqref{eq:eta_in_terms_of_third_der},
    \begin{equation*}
        \frac{K_{q,r}\lVert S\rVert^2}{3}
        \frac{\eta}{\sqrt{1-\eta^2}}
        \le
        \frac{2\lambda}{3}.
    \end{equation*}
    We conclude that the factor in parentheses in~\eqref{eq:bound_below_Df} is at least $\frac{\lambda}3$, proving \eqref{eq:effective_local_taylor}.

    It remains  to obtain an explicit admissible value for $K_{q,r}$. The first summand in $D$ is quadratic so its third derivative vanishes, and hence
    \begin{equation*}
        -D'''[h](g,g,g)
        =
        \left.
            \partial_\varepsilon^3\right|_{\varepsilon=0}
            \lVert S(h+\varepsilon g)\rVert_{L^qL^r}^2
        .
    \end{equation*}
    Letting 
    $
        A(\varepsilon):=\lVert Sh+\varepsilon Sg\rVert_{L^qL^r}^q,
    $
    it holds that
    \begin{equation}\label{eq:third_derivative_computed}\notag
        \begin{split}
            -D'''[h](g,g,g)
            ={}&
            \frac{2}{q}A(0)^{\frac{2}{q}-1}A'''(0)\\
            &+
            3\frac{2}{q}
            \left(\frac{2}{q}-1\right)
            A(0)^{\frac{2}{q}-2}A'(0)A''(0)\\
            &+
            \frac{2}{q}
            \left(\frac{2}{q}-1\right)
            \left(\frac{2}{q}-2\right)
            A(0)^{\frac{2}{q}-3}A'(0)^3.
        \end{split}
    \end{equation}
    Differentiation under the integral sign and an application of Hölder's inequality together yield \begin{equation}\label{eq_Homogenous}
        \lvert A^{(j)}(0)\rvert
        \le
        a_j(q,r)
        \lVert Sh\rVert_{L^qL^r}^{q-j}
        \lVert Sg\rVert_{L^qL^r}^{j},
        \qquad
        j\in\{1,2,3\},
    \end{equation}
    for some coefficients $a_j(q,r)>0$ to be specified below. The powers on the right-hand side of \eqref{eq_Homogenous} are forced by homogeneity, as we are differentiating the $q$-homogeneous expression $\lVert Sh+\varepsilon Sg\rVert_{L^qL^r}^q$. We  infer that we can take
    \begin{equation}\label{eq:Kqr_function_aj}
        K_{q,r}
        :=
        \frac{2}{q}a_3(q,r)
        +
        \frac{6(q-2)}{q^2}a_1(q,r)a_2(q,r)
        +
        \frac{4(q-2)(q-1)}{q^3}a_1(q,r)^3.
    \end{equation}
    
    We  turn to the computation of the constants $a_j(q,r)$ for $j\in\{1,2,3\}$. Observe that 
    \begin{equation*}
        A'(0)=q\lVert Sh\rVert_{L^qL^r}^{q-1}\left. \partial_\varepsilon\right|_{\varepsilon=0} \lVert Sh +  \varepsilon Sg\rVert_{L^qL^r}\le q\lVert Sh\rVert_{L^qL^r}^{q-1}\lVert Sg\rVert_{L^qL^r}
    \end{equation*}
    because $\lVert\cdot\rVert_{L^qL^r}$ is  1-Lipschitz, so we may take $a_1(q, r):=q$. To compute $a_2(q,r)$, direct differentiation yields
    \begin{equation*}
        \begin{split}
            A''(0)
            ={}&
            q(r-2)\int_{-\infty}^\infty
            \lVert Sh(t)\rVert_{L^r}^{q-r}
            \int_{\mathbb R^d}
            |Sh|^{r-4}
            \bigl(\Re(\overline{Sh}\,Sg)\bigr)^2
            \,\d \boldsymbol{x}\,\d t\\
            &+
            q\int_{-\infty}^\infty
            \lVert Sh(t)\rVert_{L^r}^{q-r}
            \int_{\mathbb R^d}
            |Sh|^{r-2}|Sg|^2
            \,\d \boldsymbol{x}\,\d t\\
            &+
            q(q-r)\int_{-\infty}^\infty
            \lVert Sh(t)\rVert_{L^r}^{q-2r}
            \left(
                \int_{\mathbb R^d}
                |Sh|^{r-2}\Re(\overline{Sh}\,Sg)
                \,\d \boldsymbol{x}
            \right)^2 \d t.
        \end{split}
    \end{equation*}
    By Hölder's inequality, each integral is bounded by $\lVert Sh\rVert_{L^qL^r}^{q-2}
        \lVert Sg\rVert_{L^qL^r}^2$,
    so
    \begin{equation*}
        |A''(0)|
        \le
        q\bigl(r-1+|q-r|\bigr)
        \lVert Sh\rVert_{L^qL^r}^{q-2}
        \lVert Sg\rVert_{L^qL^r}^2,
    \end{equation*}
    and we may take
    $
        a_2(q,r):=q\bigl(r-1+|q-r|\bigr).
    $
    We now turn to $a_3(q, r)$. To compute $A'''(0)$, introduce
    \begin{equation*}
        B(\varepsilon,t)
        :=
        \int_{\mathbb R^d}
        |Sh(t,\boldsymbol{x})+\varepsilon Sg(t,\boldsymbol{x})|^r\,\d \boldsymbol{x}.
    \end{equation*}
    Then $A(\varepsilon)=\int_{-\infty}^\infty B(\varepsilon,t)^{q/r}\,\d t$, so differentiation yields
    \begin{equation*}
        \begin{split}
            B'''(0,t)
            ={}&
            r(r-2)(r-4)
            \int_{\mathbb R^d}
            |Sh|^{r-6}
            \bigl(\Re(\overline{Sh}\,Sg)\bigr)^3\,\d\boldsymbol{x}\\
            &+
            3r(r-2)
            \int_{\mathbb R^d}
            |Sh|^{r-4}
            \Re(\overline{Sh}\,Sg)|Sg|^2\,\d \boldsymbol{x}.
        \end{split}
    \end{equation*}
    Thus, by Hölder's inequality,
    \begin{equation*}
        \begin{split}
            |B'(0,t)|
            &\le
            r\lVert Sh(t)\rVert_{L^r}^{r-1}
            \lVert Sg(t)\rVert_{L^r},\\
            |B''(0,t)|
            &\le
            r(r-1)\lVert Sh(t)\rVert_{L^r}^{r-2}
            \lVert Sg(t)\rVert_{L^r}^2,\\
            |B'''(0,t)|
            &\le
            r(r-2)(r+1)\lVert Sh(t)\rVert_{L^r}^{r-3}
            \lVert Sg(t)\rVert_{L^r}^3.
        \end{split}
    \end{equation*}
    Moreover,
    \begin{equation*}
    \begin{split}
        A'''(0)
        ={}&
        \frac qr\int_{-\infty}^\infty
        B(0,t)^{\frac qr-1}B'''(0,t)\,\d t\\
        &+
        3\frac qr\left(\frac qr-1\right)
        \int_{-\infty}^\infty
        B(0,t)^{\frac qr-2}B'(0,t)B''(0,t)\,\d t\\
        &+
        \frac qr\left(\frac qr-1\right)
        \left(\frac qr-2\right)
        \int_{-\infty}^\infty
        B(0,t)^{\frac qr-3}B'(0,t)^3\,\d t.
    \end{split}
    \end{equation*}
    Applying Hölder's inequality one last time yields
    \begin{equation*}
        |A'''(0)|
        \le
        a_3(q,r)
        \lVert Sh\rVert_{L^qL^r}^{q-3}
        \lVert Sg\rVert_{L^qL^r}^3,
    \end{equation*}
    and thus we may take
    \begin{equation*}
        a_3(q,r)
        :=
        q\left(
            (r-2)(r+1)
            +3|q-r|(r-1)
            +|q-r||q-2r|
        \right).
    \end{equation*}
    Inserting these values for $a_j(q, r)$ into~\eqref{eq:Kqr_function_aj} reveals that \eqref{eq:K_explicit} is indeed a valid choice for $K_{q, r}$. This concludes the proof of the lemma. 
\end{proof}

We now come to the proof of  Theorem~\ref{thm:abstract_stability}. Let $f\in\mathcal H\setminus\{0\}$. Since $\Phi_0=\operatorname{Id}$, property~(i) forces $\overline{\mathcal O}\subseteq M$. Conversely, if $f\in M\setminus\overline{\mathcal O}$, the fact that $D[\Phi_s(f)]$ strict decreases in $s$ yields $D[\Phi_s(f)]<D[f]=0$ for sufficiently small $s>0$.
This contradicts $D\ge0$,  hence $M=\overline{\mathcal O}$. In particular,
$d(f,M)=d(f,\overline{\mathcal O})=d(f,\mathcal O)$.
The conclusion of the theorem is immediate if $d(f,M)=0$, so suppose  $d(f,M)>0$. If  $d(f, M)\le \eta \lVert f\rVert$, then by Lemma~\ref{lem:third_estimate} we have
\begin{equation}\label{eq:obvious_close_estimate}
    \frac{D[f]}{d^2(f, M)}\ge \frac{\lambda}{3},
\end{equation}
and there is nothing to prove. So suppose that $f$ is far from $M$ in the sense that 
$    d(f, M)>\eta \lVert f\rVert.$
By property~(i), there exists $s$ such that
\begin{equation}\label{eq:f_can_be_made_close}
    d(\Phi_s(f), M)=\eta \left\lVert \Phi_s(f)\right\rVert.
\end{equation}
Since $d(f,M)\le\lVert f\rVert$, the norm preservation, the flow  monotonicity, Lemma~\ref{lem:third_estimate} and identity
\eqref{eq:f_can_be_made_close} together yield
\begin{equation}\label{eq:proof_explicit_stability}
    \begin{split}
        \frac{D[f]}{d(f,M)^2}
        \ge \frac{D[f]}{\lVert f\rVert^2}
        \ge \frac{D[\Phi_s(f)]}{\lVert\Phi_s(f)\rVert^2}
        \ge \frac{\lambda}{3}\frac{d(\Phi_s(f),M)^2}{\lVert\Phi_s(f)\rVert^2}
        =\frac{\lambda\eta^2}{3}.
    \end{split}
\end{equation}
This establishes~\eqref{eq:explicit_stability}, and concludes the proof of the theorem.

\subsection{Proof of Theorem~\ref{thm:effective_stability}}

Now that Theorem~\ref{thm:abstract_stability} is established, it remains to apply it to prove Theorem~\ref{thm:effective_stability}. The deficit functionals under consideration are 
\begin{equation*}
    \begin{split}
        \delta_{6,6}[f]
        &=
        \frac{1}{\sqrt[6]{3}}\lVert f\rVert_{L^2(\mathbb R)}^2
        -
        \lVert e^{-it\Delta/2}f\rVert_{L^6_tL^6_x(\R^{1+1})}^2,\\
        \delta_{8,4}[f]
        &=
        \frac{1}{\sqrt[4]{2}}\lVert f\rVert_{L^2(\mathbb R)}^2
        -
        \lVert e^{-it\Delta/2}f\rVert_{L^8_tL^4_x(\R^{1+1})}^2,\\
        \delta_{4,4}[f]
        &=
        \frac{1}{\sqrt{2}}\lVert f\rVert_{L^2(\mathbb R^2)}^2
        -
        \lVert e^{-it\Delta/2}f\rVert_{L^4_tL^4_{\boldsymbol{x}}(\R^{1+2})}^2.
    \end{split}
\end{equation*}
The corresponding spectral gaps $\Lambda_{q,r}(d)$ were recorded in~\eqref{eq:spectral_gaps_concrete}. The operator norms of $S=e^{-it\Delta/2}$, defined  in  \eqref{eq_Strichartz}, are
\begin{equation*}
    {\bf S}_{6,6}^2=\frac{1}{\sqrt[6]{3}},
    \qquad
    {\bf S}_{8,4}^2=\frac{1}{\sqrt[4]{2}},
    \qquad
    {\bf S}_{4,4}^2=\frac{1}{\sqrt{2}};
\end{equation*}
see \S\ref{sec_context}.
Finally, the explicit formula~\eqref{eq:K_explicit} yields
\begin{equation*}
    K_{6,6}=256,
    \qquad
    K_{8,4}=512,
    \qquad
    K_{4,4}=80.
\end{equation*}
Inserting these values into~\eqref{eq:eta_in_terms_of_third_der} yields
\begin{equation*}
    \begin{split}
        \eta_{6,6}
        &=
        \min\left\{
            \frac{1}{\sqrt{5}},
            \frac{2\Lambda_{6,6}(1)}
            {\sqrt{K_{6,6}^2{\bf S}_{6,6}^4
            +4\Lambda_{6,6}(1)^2}}
        \right\}
        =
        \frac{1}{\sqrt{331777}},\\
        \eta_{8,4}
        &=
        \min\left\{
            \frac{1}{\sqrt{5}},
            \frac{2\Lambda_{8,4}(1)}
            {\sqrt{K_{8,4}^2{\bf S}_{8,4}^4
            +4\Lambda_{8,4}(1)^2}}
        \right\}
        =
        \frac{3}{\sqrt{4194313}},\\
        \eta_{4,4}
        &=
        \min\left\{
            \frac{1}{\sqrt{5}},
            \frac{2\Lambda_{4,4}(2)}
            {\sqrt{K_{4,4}^2{\bf S}_{4,4}^4
            +4\Lambda_{4,4}(2)^2}}
        \right\}
        =
        \frac{1}{\sqrt{25601}}.
    \end{split}
\end{equation*}
Consequently,
\begin{equation*}
    \frac{\Lambda_{6,6}(1)\eta_{6,6}^2}{3}
    =
    \frac{2}{8957979\sqrt[6]{3}},
    \qquad
    \frac{\Lambda_{8,4}(1)\eta_{8,4}^2}{3}
    =
    \frac{9}{33554504\sqrt[4]{2}},
\end{equation*}
and
\begin{equation*}
    \frac{\Lambda_{4,4}(2)\eta_{4,4}^2}{3}
    =
    \frac{1}{307212\sqrt{2}},
\end{equation*}
as prescribed in the statement of Theorem~\ref{thm:effective_stability}.
In order to apply Theorem~\ref{thm:abstract_stability}, we start by
noting that, in all cases under consideration, $q$ and $r$ are even integers, which forces
\begin{equation*}
    \lVert e^{-it\Delta/2}f\rVert_{L^q_tL^r_{\boldsymbol{x}}(\R^{1+d})}
    \leq
    \lVert e^{-it\Delta/2}\lvert f\rvert\rVert_{L^q_tL^r_{\boldsymbol{x}}(\R^{1+d})},
\end{equation*}
and so
$\delta_{q,r}[f]\geq\delta_{q,r}[\lvert f\rvert]$.
Moreover, $r$ divides $q$, which puts us in a position to exploit the
heat-flow monotonicity from~\cite{BBCH09}. More precisely, set
\begin{equation*}
    \Phi_s(f):=\sqrt{e^{s\Delta}\lvert f\rvert^2}.
\end{equation*}
Then $\Phi_s$ preserves the $L^2$-norm, and the map
\begin{equation*}
    s\mapsto
    \lVert
        e^{-it\Delta/2}\Phi_s(f)
    \rVert_{L^q_tL^r_{\boldsymbol{x}}(\R^{1+d})}
\end{equation*}
is nondecreasing. The equality cases and the limiting behavior
described in~\cite{BBCH09} show that the corresponding deficit is
strictly decreasing away from the manifold $\Gg,$ and that
$d(\Phi_s(f),\Gg)\to0$, as $s\to\infty$. Thus $\Phi_s$ satisfies
property~{\rm (i)} of Theorem~\ref{thm:abstract_stability}, except
that $\Phi_0(f)=\lvert f\rvert$, so that $\Phi_0$ is the identity only
on the cone of nonnegative functions. Applying the proof of
Theorem~\ref{thm:abstract_stability} on this cone, with initial datum
$\lvert f\rvert$, yields the desired conclusion.

\section*{Acknowledgements}

FG  acknowledges support from the following funding agencies: The Office of Naval Research GRANT14201749 (award number N629092412126), The Serrapilheira Institute (Serra-2211-41824), FAPERJ (E-26/200.209/2023 and E-26/210.245/2024) and CNPq (309910/2023-4).
GN and DOS were funded by FCT/Portugal and the Recovery and Resilience Plan (PRR) through projects UID/04459/2025 and UID/PRR/04459/2025, and by the project 2023.17881.ICDT (SHADE).
The authors are grateful to René Quilodrán for inspiring discussions regarding the present work.

The authors acknowledge the use of AI tools for computational assistance and proofreading. All mathematical ideas, arguments, and proofs
in this work were developed, checked, and written by the authors.

\appendix

\section{Symmetry group and gaussians}\label{app_symmetry}
Let $u(t,\boldsymbol{x})=e^{-it\Delta/2}f(\boldsymbol{x})$ denote the solution to the Schr\"odinger initial
value problem \eqref{eq_freeSchr}. The following transformations map $u$ into another
solution $v(t,\boldsymbol{x})=e^{-it\Delta/2}g(\boldsymbol{x})$ of \eqref{eq_freeSchr}.
\begin{equation}\label{eq_symmetry_list}
    \begin{array}{lll}
        v(t,\boldsymbol{x})=u(t+t_0,\boldsymbol{x}), &
        g(\boldsymbol{x})=e^{-it_0\Delta/2}f(\boldsymbol{x}), &
        t_0\in\R, \\[0.4em]
        v(t,\boldsymbol{x})=u(t,\boldsymbol{x}+\boldsymbol{x}_0), &
        g(\boldsymbol{x})=f(\boldsymbol{x}+\boldsymbol{x}_0), &
        \boldsymbol{x}_0\in\R^d, \\[0.4em]
        v(t,\boldsymbol{x})=u(\lambda^2 t,\lambda \boldsymbol{x}), &
        g(\boldsymbol{x})=f(\lambda \boldsymbol{x}), &
        \lambda>0, \\[0.4em]
        v(t,\boldsymbol{x})=
        e^{i(\boldsymbol{x}\cdot\boldsymbol{\xi}_0+\frac t2|\boldsymbol{\xi}_0|^2)}
        u(t,\boldsymbol{x}+t\boldsymbol{\xi}_0), &
        g(\boldsymbol{x})=e^{i\boldsymbol{x}\cdot\boldsymbol{\xi}_0}f(\boldsymbol{x}), &
        \boldsymbol{\xi}_0\in\R^d, \\[0.4em]
        v(t,\boldsymbol{x})=cu(t,\boldsymbol{x}), &
        g(\boldsymbol{x})=cf(\boldsymbol{x}), &
        c\in\Co\setminus\{0\}.
    \end{array}
\end{equation}
These transformations generate a Lie group of real dimension $2d+4$, and
\eqref{eq_Strichartz} is invariant under all of them.
Applying \eqref{eq_symmetry_list} to the standard gaussian
$g_0(\boldsymbol{x})=e^{-|\boldsymbol{x}|^2/2}$ produces exactly the family \eqref{eq_gaussian_family_intro}. Indeed,
\begin{equation}\label{eq_PropGauss}
(e^{-it\Delta/2}g_0)(\boldsymbol{x})
=
\frac{1}{(1-it)^{d/2}}
e^{-\frac{|\boldsymbol{x}|^2}{2(1-it)}},
\end{equation}
so combining time translation and scaling yields
\[
\begin{aligned}
(e^{-it_0\Delta/2}g_0)(\lambda \boldsymbol{x})
&=
\frac{1}{(1-it_0)^{d/2}}
\exp\!\left(-\frac{\lambda^2|\boldsymbol{x}|^2}{2(1-it_0)}\right)\\
&=
\frac{1}{(1-it_0)^{d/2}}
\exp\!\left(
-\left(
\frac{\lambda^2}{2(1+t_0^2)}
+i\frac{\lambda^2t_0}{2(1+t_0^2)}
\right)|\boldsymbol{x}|^2
\right),
\end{aligned}
\]
whose quadratic coefficient has positive real part. Spatial translations,
modulations, and multiplication by constants then give precisely
\[ \Gg
    =
    \left\{
    c\,e^{-z|\boldsymbol x-\boldsymbol x_0|^2}e^{i\boldsymbol x\cdot \boldsymbol\xi_0}
    :
    c\in\Co\setminus\{0\},\;
    \Re z>0,\;
    \boldsymbol x_0,\boldsymbol\xi_0\in\R^d
    \right\}.\]
To compute the tangent space at $g_0$, consider the parametrization
\[
\Phi(c,z,\boldsymbol{x}_0,\boldsymbol{\xi}_0)(\boldsymbol{x})
:=c\,e^{-z|\boldsymbol{x}-\boldsymbol{x}_0|^2}e^{i\boldsymbol{x}\cdot\boldsymbol{\xi}_0}.
\]
At $(c,z,\boldsymbol{x}_0,\boldsymbol{\xi}_0)=(1,\frac12,0,0)$, its differential over the real parameters is
\[
D\Phi(\dot c,\dot z,\dot{ \boldsymbol{x}}_0,\dot{\boldsymbol{\xi}}_0)
=
\bigl(\dot c-\dot z|\boldsymbol{x}|^2+\boldsymbol{x}\cdot\dot{\boldsymbol{x}}_0
+i\boldsymbol{x}\cdot\dot{\boldsymbol{\xi}}_0\bigr)g_0,
\]
where $\dot c,\dot z\in\Co$ and $\dot {\boldsymbol{x}}_0,\dot{\boldsymbol{\xi}}_0\in\R^d$. Therefore
\[
T_{g_0}\Gg
=
\operatorname*{span}_{\Co}
\{g_0,x_1g_0,\ldots,x_dg_0,|\boldsymbol{x}|^2g_0\},
\]
{as claimed.


\section{PARI/GP code for the polynomial inequalities}\label{app_code}

\begin{verbatim}

/* x=lambda, nu=d/2-1, F(a,k)=F_{k,a}(x) */
F(a,k)=sum(j=0,k,binomial(k+a,k-j)*binomial(k,j)*x^(2*j)*(1-x)^(2*k-2*j));
p(k,m)=F(nu+m,k)*x^(m-1)+\
((nu+2)/(nu+1)-1/x)*(1-x)^(2*k-1)*binomial(nu+k,k)*(m==0);

/* three-term recurrence */
X=(x^2+(1-x)^2)/(1-2*x);
A(a,n)=(1-2*x)*(2*n+a)*(2*n+a-1)/(2*n*(n+a));
B(a,n)=(1-2*x)*a^2*(2*n+a-1)/(2*n*(n+a)*(2*n+a-2));
C(a,n)=(1-2*x)^2*(2*n+a)*(n-1)*(n+a-1)/(n*(n+a)*(2*n+a-2));

vector(10,n,F(a,n+1)==(A(a,n+1)*X+B(a,n+1))*F(a,n)\
-C(a,n+1)*F(a,n-1))==vector(10,n,1)

/* induction computations */
T(a,n)=(a+2*n)*(a^2+a+2*n-2)*x-a^2*(a+2*n-1);

A(a,n)*X+B(a,n)-C(a,n)==\
1-(1-x)*T(a,n)/(n*(a+2*n-2)*(a+n))

subst(T(a,n),x,a/(a+2))==4*n*(n-1)*a/(a+2)

subst(T(a,n),x,(a+1)/(a+3))==\
(2*a^2+(4*n^2-2)*(a+1/2)+2*n^2+1-4*n)/(a+3)

subst(subst(T(a,n),x,(a+1)/(a+3)),n,2)==\
(2*a^2+14*a+8)/(a+3)

/* dimensions d=1,2 */
(1-p(1,1))/(1-x)==(nu+3)*x-nu-1

subst((p(1,1)-p(2,0))/(1-x),nu,-1/2)==\
13/4*x^2-3*x+3/4

poldisc(13/4*x^2-3*x+3/4)==-3/4

subst((p(1,1)-p(2,0))/(1-x),nu,0)==(2*x-1)^2

p(0,2)==x

/* dimensions d=1,...,5 */
(1-p(2,0))/(1-x)==\
(3*nu/2+4)*x^2-(nu+1)*x-nu/2

subst((1-p(2,0))/(1-x),nu,-1/2)==\
13/4*x^2-x/2+1/4

subst((1-p(2,0))/(1-x),nu,0)==\
4*x*(x-1/4)

subst((1-p(2,0))/(1-x),nu,1/2)==\
lift(Mod(19^2/76*(x-(3+2*z)/19)*(x-(3-2*z)/19),z^2-7))

subst((1-p(2,0))/(1-x),nu,1)==\
lift(Mod(11^2/22*(x-(2+z)/11)*(x-(2-z)/11),z^2-15))

subst((1-p(2,0))/(1-x),nu,3/2)==\
25/4*(x-3/5)*(x+1/5)

/* remaining claim for d=1,...,9 */
g1=(x*p(1,1)-F(nu,4))/(1-x);
g2=(p(2,0)-p(3,0))/(1-x);

subst(g1,x,1)==5
subst(g2,x,1)==2

/* Sturm is needed for g1 only when d=1,...,4 */
apply(d->polsturm(subst(g1,nu,d/2-1),\
max((d-2)/d,d/(d+4)),oo),[1..4])==vector(4,j,0)

/* for g1 and d=5,...,9, the shifted coefficients are positive */
apply(d->vecmin(Vec(substvec(g1,[nu,x],\
[d/2-1,y+(d-2)/d])))>0,[5..9])==vector(5,j,1)

/* Sturm is needed for g2 only when d=1,...,8 */
apply(d->polsturm(subst(g2,nu,d/2-1),\
max((d-2)/d,d/(d+4)),oo),[1..8])==vector(8,j,0)

/* for g2 and d=9, the shifted coefficients are positive */
vecmin(Vec(substvec(g2,[nu,x],[7/2,y+7/9])))>0

/* remaining claim for d>=10 */
Vec(subst(g1,x,y+nu/(nu+1)))==\
[\
(nu^4+26*nu^3+251*nu^2+1066*nu+1680)/24,\
(25*nu^4+426*nu^3+2243*nu^2+2682*nu-5040)/(24*(nu+1)),\
(16*nu^5+437*nu^4+2946*nu^3+3607*nu^2-9598*nu+6960)/(24*(nu+1)^2),\
(208*nu^5+2205*nu^4+3122*nu^3-10929*nu^2+11826*nu-5520)/(24*(nu+1)^3),\
(72*nu^6+1240*nu^5+2787*nu^4-6210*nu^3+11673*nu^2-7114*nu+2784)/(24*(nu+1)^4),\
(336*nu^6+1208*nu^5-2957*nu^4+4782*nu^3-5911*nu^2+2182*nu-936)/(24*(nu+1)^5),\
(72*nu^7+480*nu^6-352*nu^5+1871*nu^4-1090*nu^3+\
1741*nu^2-202*nu+216)/(24*(nu+1)^6),\
(48*nu^7-48*nu^6+200*nu^5-169*nu^4+198*nu^3-155*nu^2-2*nu-24)/(24*(nu+1)^7)\
]

Vec(subst(g2,x,y+nu/(nu+1)))==\
[\
(4*nu^2+29*nu+54)/3,\
(9*nu^3+76*nu^2+113*nu-186)/(6*(nu+1)),\
(3*nu^3+2*nu^2-41*nu+18)/(nu+1)^2,\
(6*nu^4+27*nu^3-50*nu^2+143*nu-18)/(6*(nu+1)^3),\
(3*nu^4-3*nu^3+13*nu^2-10*nu)/(3*(nu+1)^4)\
]

/* after nu=mu+4, all numerator and denominator coefficients are nonnegative */
apply(c->\
vecmin(Vec(subst(numerator(c),nu,mu+4)))>=0\
&&vecmin(Vec(denominator(c)))>=0,\
concat(\
Vec(subst(g1,x,y+nu/(nu+1))),\
Vec(subst(g2,x,y+nu/(nu+1)))\
))==vector(13,j,1)
\end{verbatim}

\end{document}